\documentclass[12pt,oneside,english,reqno]{amsart}
\usepackage{charter}
\usepackage[T1]{fontenc}
\usepackage[latin9]{inputenc}
\usepackage{geometry}
\usepackage{color}
\usepackage{babel}
\usepackage{verbatim}
\usepackage{prettyref}
\usepackage{amsthm}
\usepackage{amstext}
\usepackage{amssymb}
\usepackage{setspace}
\usepackage{esint}
\usepackage[numbers]{natbib}
\usepackage[unicode=true,
 bookmarks=false,
 breaklinks=false,pdfborder={0 0 1},backref=page,colorlinks=true]
 {hyperref}
\hypersetup{pdftitle={Ergodic theory for quantum semigroups},
 pdfauthor={Volker Runde and Ami Viselter},
 pdfstartview={XYZ null null 1.25}}

\makeatletter
\numberwithin{equation}{section}
\theoremstyle{plain}
\newtheorem{thm}{\protect\theoremname}[section]
  \theoremstyle{definition}
  \newtheorem{defn}[thm]{\protect\definitionname}
  \theoremstyle{remark}
  \newtheorem{rem}[thm]{\protect\remarkname}
  \theoremstyle{definition}
  \newtheorem{example}[thm]{\protect\examplename}
  \theoremstyle{plain}
  \newtheorem{lem}[thm]{\protect\lemmaname}
  \theoremstyle{plain}
  \newtheorem{cor}[thm]{\protect\corollaryname}
  \theoremstyle{plain}
  \newtheorem{prop}[thm]{\protect\propositionname}
  \theoremstyle{definition}
  \newtheorem{problem}[thm]{\protect\problemname}
  \theoremstyle{remark}
  \newtheorem*{acknowledgement*}{\protect\acknowledgementname}

\usepackage{eucal}
\let\mathcal=\CMcal
\usepackage{dsfont}

\let\ldash=\l

\DeclareFontFamily{U}{mathx}{\hyphenchar\font45}
\DeclareFontShape{U}{mathx}{m}{n}{
      <5> <6> <7> <8> <9> <10>
      <10.95> <12> <14.4> <17.28> <20.74> <24.88>
      mathx10
      }{}
\DeclareSymbolFont{mathx}{U}{mathx}{m}{n}
\DeclareFontSubstitution{U}{mathx}{m}{n}
\DeclareMathAccent{\widecheck}{0}{mathx}{"71}
\renewcommand{\theenumi}{(\alph{enumi})}
\renewcommand{\labelenumi}{\theenumi}

\newrefformat{eq}{\eqref{#1}}
\newrefformat{def}{Definition \ref{#1}}
\newrefformat{lem}{Lemma \ref{#1}}
\newrefformat{pro}{Proposition \ref{#1}}
\newrefformat{prop}{Proposition \ref{#1}}
\newrefformat{cor}{Corollary \ref{#1}}
\newrefformat{thm}{Theorem \ref{#1}}
\newrefformat{exa}{Example \ref{#1}}
\newrefformat{rem}{Remark \ref{#1}}
\newrefformat{sec}{\S \ref{#1}}
\newrefformat{sub}{Subsection \ref{#1}}
\newrefformat{conj}{Conjecture \ref{#1}}

\makeatother

  \providecommand{\acknowledgementname}{Acknowledgement}
  \providecommand{\corollaryname}{Corollary}
  \providecommand{\definitionname}{Definition}
  \providecommand{\examplename}{Example}
  \providecommand{\lemmaname}{Lemma}
  \providecommand{\problemname}{Problem}
  \providecommand{\propositionname}{Proposition}
  \providecommand{\remarkname}{Remark}
\providecommand{\theoremname}{Theorem}

\begin{document}
\global\long\def\e{\varepsilon}
\global\long\def\N{\mathbb{N}}
\global\long\def\Z{\mathbb{Z}}
\global\long\def\Q{\mathbb{Q}}
\global\long\def\R{\mathbb{R}}
\global\long\def\C{\mathbb{C}}
\global\long\def\G{\mathbb{G}}
\global\long\def\HH{\mathbb{H}}
\global\long\def\norm#1{\left\Vert #1\right\Vert }
\global\long\def\cbnorm#1{\left\Vert #1\right\Vert _{\mathrm{cb}}}
\global\long\def\cco{\overline{\mathrm{co}}}

\global\long\def\H{\EuScript H}
\global\long\def\K{\EuScript K}
\global\long\def\a{\alpha}
\global\long\def\be{\beta}
\global\long\def\l{\lambda}
\global\long\def\z{\zeta}
\global\long\def\Aa{\mathcal{A}}

\global\long\def\Img{\operatorname{Im}}
\global\long\def\linspan{\operatorname{span}}
\global\long\def\slim{\operatorname*{s-lim}}
\global\long\def\LIM{\operatorname*{LIM}}
\global\long\def\cco{\operatorname{\overline{co}}}
\global\long\def\clinspan{\operatorname{\overline{span}}}
\global\long\def\ev{\operatorname{ev}}

\global\long\def\tensor{\otimes}
\global\long\def\tensormin{\otimes_{\mathrm{min}}}
\global\long\def\tensorn{\overline{\otimes}}
\global\long\def\tensori{\widecheck\otimes}
\global\long\def\tensorp{\widehat{\otimes}}
\global\long\def\tensorf{\overline{\otimes}_{\mathcal{F}}}
\global\long\def\qtensor{\xymatrix{*+<.7ex>[o][F-]{\scriptstyle \top}}
 }

\global\long\def\A{\forall}

\global\long\def\i{\mathrm{id}}

\global\long\def\one{\mathds{1}}
\global\long\def\tr{\mathrm{tr}}

\global\long\def\om{\omega}
\global\long\def\Linfty#1{L^{\infty}(#1)}
\global\long\def\Lone#1{L^{1}(#1)}
\global\long\def\Ad#1{\mathrm{Ad}(#1)}
\global\long\def\Mor#1{\mathrm{Mor}(#1)}
\global\long\def\AP{\mathrm{AP}}
\global\long\def\CAP{\mathrm{CAP}}
\global\long\def\CB{\mathcal{CB}}
\global\long\def\CK{\mathcal{CK}}

\title{Ergodic theory for quantum semigroups}

\author{Volker Runde}

\address{Department of Mathematical and Statistical Sciences, University of
Alberta, Edmonton, Alberta T6G 2G1, Canada}

\email{vrunde@ualberta.ca}

\author{Ami Viselter}

\address{Department of Mathematical and Statistical Sciences, University of
Alberta, Edmonton, Alberta T6G 2G1, Canada}

\email{viselter@ualberta.ca}

\thanks{Both authors were supported by NSERC Discovery Grants.}


\begin{abstract}
Recent results of L.~Zsid\'{o}, based on his previous work with
C.~P.~Niculescu and A.~Str\"{o}h, on actions of topological semigroups
on von Neumann algebras, give a Jacobs--de~Leeuw--Glicksberg splitting
theorem at the von Neumann algebra (rather than Hilbert space) level.
We generalize this to the framework of actions of quantum semigroups,
namely Hopf--von Neumann algebras. To this end, we introduce and study
a notion of almost periodic vectors and operators that is suitable
for our setting.
\end{abstract}
\maketitle

\section*{Introduction}

The celebrated Jacobs--de~Leeuw--Glicksberg splitting theorem \citep{Jacobs_ergodic_thy_1956,deLeeuw_Glicksberg__appl_almst_per_com}
is a fundamental result in ergodic theory. It considers a weakly almost
periodic semigroup $\mathcal{S}$ of operators over a Banach space
$X$, and under an amenability condition, it gives a decomposition
of $X$ as the direct sum of the \emph{almost periodic} vectors of
$\mathcal{S}$ and the \emph{weakly mixing} (or \emph{flight}) vectors
of $\mathcal{S}$. Assume now that $N$ is a von Neumann algebra,
$\om$ is a faithful normal state of $N$, $G$ is a topological semigroup
and $\a=\left(\a_{s}\right)_{s\in G}$ is an $\om$-preserving action
of $G$ on $N$. The quadruple $(N,\om,G,\a)$ is a \emph{noncommutative
dynamical system}, generalizing the standard ergodic-theoretic setting
of a measure-preserving transformation $T$ acting on a probability
space $(\Omega,\mu)$, giving rise to the action $f\mapsto f\circ T$,
$f\in\Linfty{\Omega,\mu}$. On the GNS Hilbert space $\H$ of $(N,\om)$,
one constructs the canonical semigroup $\left(U_{s}\right)_{s\in G}$
of isometries implementing $\a$. The Jacobs--de~Leeuw--Glicksberg
theorem now applies to $\mathcal{S}:=\left\{ U_{s}:s\in G\right\} $,
yielding a decomposition of $\H$. But in this setting, it is also
desirable to obtain a decomposition of $N$ as a direct sum of almost
periodic operators and weakly mixing operators. This problem was considered
by Niculescu, Str\"{o}h and Zsid\'{o} \citep{Niculescu_Stroh_Zsido__noncmt_recur}
in the case $G=\Z_{+}$, and by Zsid\'{o} \citep{Zsido__splitting_noncomm_dyn_sys}
in the general case where $G$ is a locally compact unital semigroup.
They define the notion of almost periodicity of an operator in $N$,
prove that the set of almost periodic operators forms a von Neumann
subalgebra $N^{\AP}$ of $N$, and establish the existence of an $\om$-preserving
conditional expectation from $N$ onto $N^{\AP}$ \citep[Theorem 4.2]{Niculescu_Stroh_Zsido__noncmt_recur}.
The decomposition of $N$ given by this conditional expectation is
the desired one.

Extending classical results of ergodic theory to noncommutative dynamical
systems has been a central research theme for many years. In addition
to the references mentioned above, we refer the reader to Abadie and
Dykema \citep{Abadie_Dykema__unique_erg}, Austin, Eisner and Tao
\citep{Austin_Eisner_Tao_noncomm_erg_av}, Beyers, Duvenhage and Str{\"o}h
\citep{Beyers_Duvenhage_Stroh__Szemeredi}, Duvenhage \citep{Duvenhage__erg_thm_quantum,Duvenhage__ergodicity_and_mixing,Duvenhage__relatively_ind_join},
Fidaleo and Mukhamedov \citep{Fidaleo_Mukhamedov__erg_prop} and the
references therein for (a partial list of) recent works in this spirit.
The purpose of the present paper is to extend the results of \citep{Niculescu_Stroh_Zsido__noncmt_recur,Zsido__splitting_noncomm_dyn_sys}
detailed in the previous paragraph to actions of quantum semigroups,
namely Hopf--von Neumann algebras. In this setting the problems become
much more delicate. For instance, it is not obvious at first what
the proper definition of almost periodicity should be.

The structure of the paper is as follows. In \prettyref{sec:prelim_notation}
we give some background and establish the notation and standing hypothesis.
In \prettyref{sec:basic_erg_thy} we prove a generalized mean ergodic
theorem, extending the mean ergodic theorem for Hopf--von Neumann
algebras of Duvenhage \citep{Duvenhage__erg_thm_quantum}. This generalization
is of independent interest. \prettyref{sec:CAP} is dedicated to (complete)
almost periodicity of actions of Hopf--von Neumann algebras: after
giving some motivation, we establish the fundamental compactification
result (\prettyref{thm:compactification_CAP}) and then give an alternative
definition of complete almost periodicity (\prettyref{cor:CAP_second_def}).
The main results of the paper are obtained in \prettyref{sec:analysis_of_CAP}.
In particular, we prove that the set $N^{\CAP}$ of completely almost
periodic operators is a von Neumann subalgebra of $N$ (\prettyref{thm:N_CAP_vN_alg}),
and that under certain conditions, there is a canonical conditional
expectation from $N$ to $N^{\CAP}$ (\prettyref{cor:E_CAP}), which
provides the Jacobs--de~Leeuw--Glicksberg splitting of $N$.

\section{Preliminaries and notation\label{sec:prelim_notation}}

We shall use the notation of Effros and Ruan \citep{Effros_Ruan__book}
for operator space terminology. Throughout the paper, the symbols
$\odot$, $\tensor$, $\tensormin$, $\tensori$, $\tensorp$ and
$\tensorn$ stand for the following respective types of tensor products:
the algebraic, Hilbert space, $C^{*}$-algebraic minimal (spatial),
operator space injective, projective and normal spatial (including
von Neumann algebraic). A projection over a Hilbert space is always
orthogonal. In the context of operators over a Hilbert space, ``weak''
and ``strong'' refer to the weak and strong operator topologies,
respectively. We assume that the reader is familiar with the basics
of compact quantum groups (Woronowicz \citep{Woronowicz__symetries_quantiques},
Maes and Van Daele \citep{Maes_van_Daele__notes_CQGs}) and the Tomita--Takesaki
modular theory \citep{Stratila__mod_thy,Takesaki__Tomita_thy,Takesaki__book_vol_2}.

The basic quantum structure we consider in this paper is that of Hopf--von
Neumann algebras.
\begin{defn}[\citep{Enock_Schwartz__book}]
A \emph{Hopf--von Neumann algebra} is a pair $\G=(\Linfty{\G},\Delta)$,
where $\Linfty{\G}$ is a von Neumann algebra and $\Delta:\Linfty{\G}\to\Linfty{\G}\tensorn\Linfty{\G}$
is a co-multiplication, that is, a unital normal $*$-homomorphism
which is co-associative: $(\Delta\tensor\i)\Delta=(\i\tensor\Delta)\Delta$.
\end{defn}
If $\G$ is a Hopf--von Neumann algebra, we denote $\Lone{\G}:=\Linfty{\mbox{\ensuremath{\G}}}_{*}$.
This space is a Banach algebra with the product $*$ induced by the
predual $\Delta_{*}:\Lone{\G}\tensorp\Lone{\G}\to\Lone{\G}$, that
is, $\left(\om_{1}*\om_{2}\right)(x):=(\om_{1}\tensor\om_{2})\Delta(x)$
($\om_{1},\om_{2}\in\Lone{\G}$, $x\in\Linfty{\G}$).
\begin{defn}
A \emph{co-representation} of $\G$ is an operator $u\in C\tensorn\Linfty{\G}$,
for some von Neumann algebra $C$, satisfying
\[
(\i\tensor\Delta)(u)=u_{12}u_{13}
\]
(we use the customary leg numbering notation). If $C=B(\K)$ for some
Hilbert space $\K$, we say that $u$ is a co-representation of $\G$
on $\K$. A closed subspace $\K'$ of $\K$ is said to be invariant
under $u$ if $u(p\tensor\one)=(p\tensor\one)u(p\tensor\one)$, where
$p$ is the projection of $\K$ onto $\K'$. The restriction $u(p\tensor\one)\in B(\K')\tensorn\Linfty{\G}$
is a typical sub-representation of $\G$.
\end{defn}

\begin{defn}
A Hopf--von Neumann algebra $\G$ is called a \emph{locally compact
quantum group} (see Kustermans and Vaes \citep{Kustermans_Vaes__LCQG_C_star,Kustermans_Vaes__LCQG_von_Neumann}
or Van Daele \citep{Van_Daele__LCQGs}, and for an alternative approach
see Masuda, Nakagami and Woronowicz \citep{Masuda_Nakagami_Woronowicz}),
abbreviated LCQG, if $\Linfty{\G}$ admits a pair of normal, semi-finite,
faithful weights $\varphi,\psi$, called the left and right Haar weights,
that are left and right invariant (respectively) in the sense that
$\varphi((\omega\tensor\i)\Delta(x))=\omega(\one)\varphi(x)$ for
all $\om\in\Lone{\G}^{+}$ and $x\in M^{+}$ such that $\varphi(x)<\infty$,
and $\psi((\i\tensor\omega)\Delta(x))=\omega(\one)\psi(x)$ for all
$\om\in\Lone{\G}^{+}$ and $x\in M^{+}$ such that $\psi(x)<\infty$.
\end{defn}
The deep theory of locally compact quantum groups is not needed in
most parts of the paper, especially outside of the appendix.
\begin{defn}[\citep{Enock_Schwartz__amenable_Kac_alg,Desmedt_Quaegebeur_Vaes,Bedos_Tuset_2003}]
Let $\G$ be a Hopf--von Neumann algebra.
\begin{enumerate}
\item We say that $\G$ is \emph{amenable} if it admits a (two-sided) invariant
mean, that is, a state $m\in\Linfty{\G}^{*}$ with
\begin{equation}
m((\om\tensor\i)\Delta(x))=\om(\one)m(x)=m((\i\tensor\om)\Delta(x))\label{eq:invariant_mean}
\end{equation}
 for all $x\in\Linfty{\G}$, $\om\in\Lone{\G}$. Recall that if $\G$
is a LCQG, then the existence of either a left- or a right-invariant
mean (only one of the equalities in \prettyref{eq:invariant_mean})
is equivalent to amenability \citep[Proposition 3]{Desmedt_Quaegebeur_Vaes}.
\item We say that $\G$ is \emph{(left)} \emph{co-amenable} if $\Lone{\G}$
possesses a bounded left approximate identity.
\end{enumerate}
\end{defn}
\begin{rem}
\label{rem:amen_norm}Let $\G$ be an amenable Hopf--von Neumann algebra,
and let $m$ be an invariant mean on $\G$. One can use a standard
convexity argument (cf.~\citep[Lemma 3.12, (1)]{Tomatsu__amenable_discrete})
to show that there is a net $\left(m_{\kappa}\right)$ of states in
$\Lone{\G}$ with $m_{\kappa}\to m$ in the $\sigma(\Linfty{\G}^{*},\Linfty{\G})$-topology
and
\[
\lim_{\kappa}\left\Vert \theta*m_{\kappa}-\theta(\one)m_{\kappa}\right\Vert =0=\lim_{\kappa}\left\Vert m_{\kappa}*\theta-\theta(\one)m_{\kappa}\right\Vert \qquad(\forall\theta\in\Lone{\G}).
\]
\end{rem}
\begin{example}
\label{exa:basic_LCQGs}The two basic examples of LCQGs are the ones
that come from a locally compact group $G$ as follows.
\begin{enumerate}
\item \label{enu:comm_LCQG}Let $\Linfty{\G}$ be $\Linfty G$ and define
$\Delta$ by $(\Delta(f))(t,s):=f(ts)$ ($f\in\Linfty G$, $s,t\in G$),
employing the natural identification $\Linfty G\tensorn\Linfty G\cong\Linfty{G\times G}$.
Now the product $*$ is the convolution on $\Lone{\G}=\Lone G$. Letting
$\varphi,\psi$ be integration against a left and a right Haar measure,
respectively, one obtains the \emph{commutative LCQG} $\G$ associated
with $G$. This $\G$ is always co-amenable, and is amenable if and
only if $G$ is amenable as a group.
\item \label{enu:co_comm_LCQG}Let $\Linfty{\G}$ be $\mathrm{VN}(G)$,
the (left) von Neumann algebra of $G$. We define $\Delta$ to be
the unique unital normal $*$-homomorphism from $\mathrm{VN}(G)$
to $\mathrm{VN}(G)\tensorn\mathrm{VN}(G)$ that satisfies $\Delta(\l_{g})=\l_{g}\tensor\l_{g}$
for every $g\in G$ ($\l_{g}$ being the left translation by $g$),
and take both $\varphi,\psi$ to be the Plancherel weight of $G$
\citep{Takesaki__book_vol_2}. This gives the \emph{co-commutative
LCQG} associated with $G$. We often write $\hat{G}$ for this $\G$.
It is always amenable, and is co-amenable if and only if $G$ is amenable
as a group.
\end{enumerate}
\end{example}
\begin{defn}
An \emph{action} of a Hopf--von Neumann algebra $\G$ on a von Neumann
algebra $N$ is a unital normal $*$-homomorphism $\a:N\to N\tensorn\Linfty{\G}$
satisfying
\[
(\a\tensor\i)\a=(\i\tensor\Delta)\a.
\]

\end{defn}
\noindent \textbf{Standing hypothesis.} Throughout the paper we assume
that $\G=(\Linfty{\G},\Delta)$ is a Hopf--von Neumann algebra and
$N$ is a von Neumann algebra such that the following hold:
\begin{enumerate}
\item $\G$ is amenable and (left) co-amenable; and
\item $\a$ is an action of $\G$ on $N$, and $\om$ is a normal faithful
state of $N$ invariant under $\a$, that is,
\[
(\om\tensor\i)\a=\om(\cdot)\one.
\]
Furthermore, for some bounded left approximate identity $(\epsilon_{\l})$
of the Banach algebra $\Lone{\G}$, we have
\[
(\i\tensor\epsilon_{\l})\a(a)\to a
\]
weakly for every $a\in N$.
\end{enumerate}
Our setting evidently generalizes that of \citep{Niculescu_Stroh_Zsido__noncmt_recur}
(take $\Linfty{\G}:=\ell_{\infty}(\Z_{+})$ and define $\Delta$ as
in \prettyref{exa:basic_LCQGs}, \prettyref{enu:comm_LCQG}).
\begin{rem}
While $N$ should be countably decomposable for $\om$ to exist, there
is no other condition on $N$. We believe it should be possible to
obtain stronger results on multiple recurrence in the spirit of Austin,
Eisner and Tao \citep{Austin_Eisner_Tao_noncomm_erg_av} if $N$ is
assumed finite (also see \prettyref{rem:E_CAP_when_N_finite}).
\end{rem}
We let $\Linfty{\G}$ act on some Hilbert space $L^{2}(\G)$ (not
necessarily in standard position). We denote by $(\H,\i,\Gamma)$
the GNS construction for $(N,\om)$. Since $\omega$ is invariant
under the action $\a$, the isometry $U\in B(\H)\tensorn\Linfty{\G}$
determined by
\[
((\i\tensor\theta)(U))\Gamma(a)=\Gamma\bigl((\i\tensor\theta)\a(a)\bigr)\qquad(\A\theta\in\Lone{\G},a\in N)
\]
implements $\a$ in the sense that $\a(a)U=U(a\tensor\one)$ for all
$a\in N$. Moreover, $U$ is a co-representation of $\G$.

For every $\z\in\H$ we define a bounded operator $T_{\z}:\Lone{\G}\to\H$
by $T_{\z}(\theta):=((\i\tensor\theta)(U))\z$ ($\theta\in\Lone{\G}$).
We note that Duvenhage \citep{Duvenhage__erg_thm_quantum} denotes
$T_{\z}(\theta)$ by $\tilde{\theta}^{\a}\z$.

We shall require a Hilbert space version of $\a$ as follows. Denote
by $\H_{c}$ the column Hilbert space determined by $\H$ \citep[\S 3.4]{Effros_Ruan__book}.
Since $\Gamma\in\CB(N,\H_{c})$, we have the map $\Gamma\tensor\i\in\CB(N\tensorn\Linfty{\G},\H_{c}\tensorn\Linfty{\G})$. 
\begin{lem}
\label{lem:alpha_tilde}The operators $T_{\z}$, $\z\in\H$, are completely
bounded, and can be regarded as elements of $\H_{c}\tensorn\Linfty{\G}$.
Furthermore, the operator $\tilde{\a}:\H_{c}\to\H_{c}\tensorn\Linfty{\G}$
given by $\tilde{\a}(\z):=T_{\z}$ for all $\z\in\H$ is completely
contractive. It satisfies $\tilde{\a}\circ\Gamma=(\Gamma\tensor\i)\circ\a$
and $(\i\tensor\Delta)\tilde{\a}=(\tilde{\a}\tensor\i)\tilde{\a}$
(which belong to $\CB(N,\H_{c}\tensorn\Linfty{\G})$ and $\CB(\H_{c},\H_{c}\tensorn\Linfty{\G}\tensorn\Linfty{\G})$,
respectively). Furthermore, $(\i\tensor\epsilon_{\l})\tilde{\a}(\xi)\to\xi$
weakly for every $\xi\in\H$. \end{lem}
\begin{proof}
Recall that the operator space structure of $\H_{c}\tensorn\Linfty{\G}$
is given by the natural $w^{*}$-homeomorphic, completely isometric
embedding $\H_{c}\tensorn\Linfty{\G}\hookrightarrow\CB(\Lone{\G},\H_{c})$
\citep[Theorem 7.2.3 and Corollary 7.1.5]{Effros_Ruan__book}. For
every $\z\in\H$, the operator $\ev_{\z}:B(\H)\to\H_{c}$ given by
$T\mapsto T\z$, $T\in B(\H)$, belongs to $\CB(B(\H),\H_{c})$ and
is $w^{*}$-continuous. Furthermore, the map $\H_{c}\to\CB(B(\H),\H_{c})$
given by $\z\mapsto\ev_{\z}$ is a complete isometry. Define $\tilde{\a}:\H_{c}\to\H_{c}\tensorn\Linfty{\G}$
by $\tilde{\a}(\z):=(\ev_{\z}\tensor\i)(U)$. From the foregoing and
as $\left\Vert U\right\Vert =1$, $\tilde{\a}$ is completely contractive.
By the definition of $T_{\z}$, we clearly have $\tilde{\a}(\z)=T_{\z}$
for all $\z\in\H$, and from the definition of $U$ we obtain the
formula $\tilde{\a}(\Gamma(a))=(\Gamma\tensor\i)\a(a)$ for all $a\in N$.
Consequently, for every $a\in N$,
\[
\begin{split}(\i\tensor\Delta)\tilde{\a}(\Gamma(a)) & =(\i\tensor\Delta)(\Gamma\tensor\i)\a(a)=(\Gamma\tensor\i\tensor\i)(\i\tensor\Delta)\a(a)\\
 & =(\Gamma\tensor\i\tensor\i)(\a\tensor\i)\a(a)=(\tilde{\a}\tensor\i)(\Gamma\tensor\i)\a(a)=(\tilde{\a}\tensor\i)\tilde{\a}(\Gamma(a)).
\end{split}
\]
Hence $(\i\tensor\Delta)\tilde{\a}=(\tilde{\a}\tensor\i)\tilde{\a}$.
By assumption, $(\i\tensor\epsilon_{\l})\a(a)\to a$ weakly for all
$a\in N$, so that
\[
(\i\tensor\epsilon_{\l})\tilde{\a}(\Gamma(a))=\Gamma((\i\tensor\epsilon_{\l})\a(a))\to\Gamma(a)
\]
weakly. From the boundedness of $\left(\epsilon_{\l}\right)$ we infer
that $(\i\tensor\epsilon_{\l})\tilde{\a}(\xi)\to\xi$ weakly for every
$\xi\in\H$. This completes the proof.\end{proof}
\begin{rem}
\label{rem:alpha_tilde}For $\z\in\H$ and $\theta,\vartheta\in\Lone{\G}$,
we have
\[
T_{T_{\z}(\vartheta)}=\tilde{\a}((\i\tensor\vartheta)\tilde{\a}(\z))=(\i\tensor\i\tensor\vartheta)(\tilde{\a}\tensor\i)\tilde{\a}(\z)=(\i\tensor\i\tensor\vartheta)(\i\tensor\Delta)\tilde{\a}(\z).
\]
Thus
\[
T_{T_{\z}(\vartheta)}(\theta)=(\i\tensor\theta\tensor\vartheta)(\i\tensor\Delta)\tilde{\a}(\z)=(\i\tensor(\theta*\vartheta))\tilde{\a}(\z)=T_{\z}(\theta*\vartheta).
\]

\end{rem}
For the definition of a compact quantum group we use the $C^{*}$-algebraic
language, which is more suitable for our purposes, although there
is an equivalent von Neumann algebraic one.
\begin{defn}[\citep{Woronowicz__symetries_quantiques,Maes_van_Daele__notes_CQGs}]
\label{def:CQG}A \emph{compact quantum group} is a pair $\HH=(C(\HH),\Delta)$,
where $C(\HH)$ is a unital $C^{*}$-algebra, $\Delta:C(\HH)\to C(\HH)\tensormin C(\HH)$
is a $C^{*}$-algebraic co-multiplication, that is, a unital $*$-homomorphism
which is co-associative, i.e., $(\Delta\tensor\i)\Delta=(\i\tensor\Delta)\Delta$,
and furthermore, the sets $(C(\HH)\tensor\one)\Delta(C(\HH))$ and
$(\one\tensor C(\HH))\Delta(C(\HH))$ are norm total in $C(\HH)\tensormin C(\HH)$.
\end{defn}
If the pair $(C(\HH),\Delta)$ satisfies all these assumptions apart
from the density conditions, it is called a compact quantum semigroup.

\section{\label{sec:basic_erg_thy}Basic ergodic theory}
\begin{lem}
\label{lem:V_u}Let $C$ be a von Neumann algebra and $u\in C\tensorn\Linfty{\G}$.
Fix $\mu\in C_{*}^{+}$, and denote its GNS construction by $(\H_{\mu},\i,\Gamma_{\mu})$.
Denote $\Xi:=\Gamma_{\mu}\tensor\Gamma$. For every $\theta\in\Lone{\G}^{+}$,
the operator $V_{u}(\theta)$ given by
\[
V_{u}(\theta):\Xi(a)\mapsto\Xi\bigl[(\i\tensor\i\tensor\theta)\bigl(u_{13}\cdot(\i\tensor\a)(a)\bigr)\bigr]\qquad(\A a\in C\tensorn N)
\]
extends to an element of $B(\H_{\mu}\tensor\H)$ satisfying $\norm{V_{u}(\theta)}\leq\left\Vert u\right\Vert \norm{\theta}$.
Moreover, if $u$ is a co-representation of $\G$, then $V_{u}(\theta_{1})V_{u}(\theta_{2})=V_{u}(\theta_{1}*\theta_{2})$
for all $\theta_{1},\theta_{2}\in\Lone{\G}^{+}$.\end{lem}
\begin{proof}
Let $a,b\in C\tensorn N$. Since $\a$ is $\om$-invariant, we have
\begin{multline*}
\left|\left\langle V_{u}(\theta)\Xi(a),\Xi(b)\right\rangle \right|\\
\begin{split} & =\left|(\mu\tensor\om)\bigl[b^{*}\cdot(\i\tensor\i\tensor\theta)\left(u_{13}(\i\tensor\a)(a)\right)\bigr]\right|=\left|(\mu\tensor\om\tensor\theta)\bigl(b_{12}^{*}u_{13}(\i\tensor\a)(a)\bigr)\right|\\
 & \leq\norm{\theta}^{1/2}\norm{\Xi(b)}(\mu\tensor\om\tensor\theta)\bigl((\i\tensor\a)(a^{*})u_{13}^{*}u_{13}(\i\tensor\a)(a)\bigr)^{1/2}\\
 & \leq\norm{\theta}^{1/2}\norm{\Xi(b)}\left\Vert u\right\Vert (\mu\tensor\om\tensor\theta)\bigl((\i\tensor\a)(a^{*}a)\bigr)^{1/2}=\left\Vert \theta\right\Vert \norm{\Xi(b)}\norm{\Xi(a)}\left\Vert u\right\Vert
\end{split}
\end{multline*}
by the Cauchy--Schwarz inequality. This proves the first assertion.
If $u$ is a co-representation of $\G$, $\theta_{1},\theta_{2}\in\Lone{\G}^{+}$
and $a\in C\tensorn N$, then 
\[
\begin{split}V_{u}(\theta_{1})V_{u}(\theta_{2})\Xi(a) & =V_{u}(\theta_{1})\Xi\bigl[(\i\tensor\i\tensor\theta_{2})\bigl(u_{13}\cdot(\i\tensor\a)(a)\bigr)\bigr]\\
 & =\Xi\left\{ (\i\tensor\i\tensor\theta_{1})\left[u_{13}\cdot(\i\tensor\a)(\i\tensor\i\tensor\theta_{2})\bigl(u_{13}\cdot(\i\tensor\a)(a)\bigr)\right]\right\} \\
 & =\Xi\left\{ (\i\tensor\i\tensor\theta_{1}\tensor\theta_{2})\left[u_{13}u_{14}\cdot(\i\tensor(\a\tensor\i)\a)(a)\right]\right\} .
\end{split}
\]
Using that $\a$ is an action and $u$ is a co-representation, we
get
\[
\begin{split}V_{u}(\theta_{1})V_{u}(\theta_{2})\Xi(a) & =\Xi\left\{ (\i\tensor\i\tensor\theta_{1}\tensor\theta_{2})\left[(\i\tensor\i\tensor\Delta)\left(u_{13}(\i\tensor\a)(a)\right)\right]\right\} \\
 & =\Xi\left[(\i\tensor\i\tensor(\theta_{1}*\theta_{2}))\left(u_{13}(\i\tensor\a)(a)\right)\right]=V_{u}(\theta_{1}*\theta_{2})\Xi(a).\\
\\
\end{split}
\]
Hence $V_{u}(\theta_{1})V_{u}(\theta_{2})=V_{u}(\theta_{1}*\theta_{2})$
by continuity.
\end{proof}
The next result is a generalized mean ergodic theorem (cf.~\citep[Proposition 3.2]{Niculescu_Stroh_Zsido__noncmt_recur})
for Hopf--von Neumann algebras. When taking $u$ to be the trivial
representation, one recovers Duvenhage's result \citep{Duvenhage__erg_thm_quantum}.
\begin{thm}
\label{thm:genr_mean_ergodic}Let $C$ be a von Neumann algebra and
$u\in C\tensorn\Linfty{\G}$ be a unitary co-representation of $\G$.
Fix $\mu\in C_{*}^{+}$, denote its GNS construction by $(\H_{\mu},\i,\Gamma_{\mu})$
and let $\Xi:=\Gamma_{\mu}\tensor\Gamma$. Denote by $P_{u}$ the
projection (over $\H_{\mu}\tensor\H$) onto the space
\[
\bigcap_{\theta\text{ is a state in }\Lone{\G}}\ker(V_{u}(\theta)-\one).
\]
Let $\left(m_{\kappa}\right)$ be a net of states in $\Lone{\G}$
such that $\left\Vert m_{\kappa}*\theta-m_{\kappa}\right\Vert \to0$
and $\theta*m_{\kappa}-m_{\kappa}\to0$ in the $\sigma(\Lone{\G},\Linfty{\G})$-topology
for each state $\theta$ in $\Lone{\G}$ (see \prettyref{rem:amen_norm}).
\begin{enumerate}
\item \label{enu:genr_mean_ergodic__1}We have $V_{u}(m_{\kappa})\to P_{u}$
strongly.
\item \label{enu:genr_mean_ergodic__2}Suppose that $\mu$ is faithful.
There exists a unique projection $E_{u}$ from $C\tensorn N$ onto
the subspace $N_{u}:=\left\{ b\in C\tensorn N:(\i\tensor\a)(b)=u_{13}^{*}b_{12}\right\} $
satisfying $\Xi\circ E_{u}=P_{u}\circ\Xi$. Furthermore, $E_{u}$
is normal, it has norm $1$ (unless $E_{u}=0$),
\[
(\i\tensor\i\tensor m_{\kappa})\left[u_{13}\cdot(\i\tensor\a)(a)\right]\to E_{u}(a)
\]
strongly for every $a\in C\tensorn N$ and $\left(m_{\kappa}\right)$
as above and $\overline{\Xi(N_{u})}=\Img P_{u}$.
\end{enumerate}
\end{thm}
\begin{rem}
We emphasize that neither $P_{u}$ nor $E_{u}$ depend on the chosen
net $\left(m_{\kappa}\right)$.\end{rem}
\begin{proof}[Proof of \prettyref{thm:genr_mean_ergodic}]
\prettyref{enu:genr_mean_ergodic__1} Recall that for any contraction
$v\in B(\H)$ we have $\ker(v-\one)=\ker(v^{*}-\one)$, thus $\ker(v-\one)^{\perp}=\overline{\Img(v-\one)}$.
Therefore $\Img(\one-P_{u})=\clinspan\bigcup_{\theta}\Img(V_{u}(\theta)-\one)$,
where the union goes over all states in $\Lone{\G}$. For every state
$\theta\in\Lone{\G}$, we have $V_{u}(m_{\kappa})\left(V_{u}(\theta)-\one\right)=V_{u}(m_{\kappa}*\theta-m_{\kappa})$,
and hence $\norm{V_{u}(m_{\kappa})\left(V_{u}(\theta)-\one\right)}\leq\left\Vert m_{\kappa}*\theta-m_{\kappa}\right\Vert \to0$.
Thus $V_{u}(m_{\kappa})\eta\to0$ for every $\eta$ in the dense subspace
$\linspan\bigcup_{\theta}\Img(V_{u}(\theta)-\one)$ of $\Img(\one-P_{u})$.
Since $\left(V_{u}(m_{\kappa})\right)_{\kappa}$ is uniformly bounded
by $1$, we deduce that $V_{u}(m_{\kappa})\eta\to0$ for all $\eta\in\Img(\one-P_{u})$.
Now, for all $\zeta\in\H_{\mu}\tensor\H$, we have
\[
V_{u}(m_{\kappa})\zeta=V_{u}(m_{\kappa})P_{u}\zeta+V_{u}(m_{\kappa})(\one-P_{u})\zeta=P_{u}\zeta+V_{u}(m_{\kappa})(\one-P_{u})\zeta\to P_{u}\zeta.
\]

\prettyref{enu:genr_mean_ergodic__2} Let $a\in C\tensorn N$ be given.
The net $\left((\i\tensor\i\tensor m_{\kappa})\bigl(u_{13}\cdot(\i\tensor\a)(a)\bigr)\right)_{\l}$
is bounded by $\norm a$, so it admits a weak cluster point in $C\tensorn N$,
denoted $E_{u}(a)$, with $\norm{E_{u}(a)}\leq\norm a$. By \prettyref{enu:genr_mean_ergodic__1}
we have
\begin{equation}
\Xi\left[(\i\tensor\i\tensor m_{\kappa})\bigl(u_{13}\cdot(\i\tensor\a)(a)\bigr)\right]=V_{u}(m_{\kappa})\Xi(a)\to P_{u}\Xi(a),\label{eq:using_genr_MEG_for_cond_exp}
\end{equation}
from which the identity $\Xi(E_{u}(a))=P_{u}\Xi(a)$ is obtained using
\prettyref{enu:genr_mean_ergodic__1}. This proves that $E_{u}$ does
not depend on the particularly chosen net $\left(m_{\kappa}\right)$
and that it is normal. Moreover, \prettyref{eq:using_genr_MEG_for_cond_exp}
entails that actually $(\i\tensor\i\tensor m_{\kappa})\bigl(u_{13}\cdot(\i\tensor\a)(a)\bigr)\to E_{u}(a)$
strongly, as we have convergence at the vector $\Xi(\one)$, which
is separating for $C\tensorn N$ on $\H_{\mu}\tensor\H$.

We now prove that $E_{u}(a)\in N_{u}$. As $\a$ is an action, we
have
\[
\begin{split}(\i\tensor\a)(E_{u}(a)) & =\lim_{\kappa}(\i\tensor\i\tensor\i\tensor m_{\kappa})(\i\tensor\a\tensor\i)\left[u_{13}\cdot(\i\tensor\a)(a)\right]\\
 & =\lim_{\kappa}(\i\tensor\i\tensor\i\tensor m_{\kappa})\left[u_{14}\cdot(\i\tensor(\a\tensor\i)\a)(a)\right]\\
 & =\lim_{\kappa}(\i\tensor\i\tensor\i\tensor m_{\kappa})\left[u_{14}\cdot(\i\tensor(\i\tensor\Delta)\a)(a)\right]
\end{split}
\]
(all limits are weak ones). Thus, if $\theta\in\Lone{\G}$ is a state,
then by approximate invariance of $(m_{\kappa})$,
\[
\begin{split}(\i\tensor\i\tensor\theta)\left[u_{13}\cdot(\i\tensor\a)(E_{u}(a))\right] & =\lim_{\kappa}(\i\tensor\i\tensor\theta\tensor m_{\kappa})\left[u_{13}u_{14}\cdot(\i\tensor(\i\tensor\Delta)\a)(a)\right]\\
 & =\lim_{\kappa}(\i\tensor\i\tensor\theta\tensor m_{\kappa})(\i\tensor\i\tensor\Delta)\left[u_{13}(\i\tensor\a)(a)\right]\\
 & =\lim_{\kappa}(\i\tensor\i\tensor(\theta*m_{\kappa}))\left[u_{13}(\i\tensor\a)(a)\right]\\
 & =\lim_{\kappa}(\i\tensor\i\tensor m_{\kappa})\left[u_{13}(\i\tensor\a)(a)\right]=E_{u}(a).
\end{split}
\]
Hence $u_{13}\cdot(\i\tensor\a)(E_{u}(a))=E_{u}(a)\tensor\one$, and
$u$ being a unitary, we obtain $E_{u}(a)\in N_{u}$. On the other
hand, we obviously have $E_{u}(a)=a$ if $a\in N_{u}$. We conclude
that $E_{u}$ is a projection from $C\tensorn N$ onto $N_{u}$. 

The inclusion $\Xi(N_{u})\subseteq\Img P_{u}$ is clear. To prove
that $\Img P_{u}\subseteq\overline{\Xi(N_{u})}$, let $\zeta\in\Img P_{u}$
and $\e>0$ be given. Pick $a\in C\tensorn N$ with $\norm{\zeta-\Xi(a)}\leq\e$.
Then $\zeta-\Xi(E_{u}(a))=P_{u}(\zeta-\Xi(a))$ has norm $\leq\e$
and $E_{u}(a)\in N_{u}$.
\end{proof}

\section{\label{sec:CAP}Complete almost periodicity}

The aim of this section is to give a definition of almost periodic
vectors and operators that is adequate for our setting of actions
of Hopf--von Neumann algebras, and then generalize the compactification
result \citep[Lemma 4.1]{Niculescu_Stroh_Zsido__noncmt_recur}. This
is an essential step for the rest of the paper.

The question of how to define almost periodicity in the quantum setting
is a nontrivial one. We begin by recalling the classical setting of
Niculescu, Str\"{o}h and Zsid\'{o} \citep{Niculescu_Stroh_Zsido__noncmt_recur}
and Zsid\'{o} \citep{Zsido__splitting_noncomm_dyn_sys}. Let $N$
be a von Neumann algebra, $\om\in N_{*}$ a faithful state, $G$ a
locally compact unital semigroup, and $\a=\left(\a_{s}\right)_{s\in G}$
a (weakly continuous) action of $G$ on $N$. We denote the GNS construction
for $(N,\om)$ by $(\H,\Gamma)$. For $s\in G$, the isometry $U_{s}\in B(\H)$
determined by $U_{s}\Gamma(a):=\Gamma(\a_{s}(a))$, $a\in N$, implements
$\a_{s}$ in the sense that $\a_{s}(a)U_{s}=U_{s}a$ for all $a\in N$.
Moreover, $U=\left(U_{s}\right)_{s\in G}$ is a representation of
$G$: $U_{s}U_{t}=U_{st}$ for every $s,t\in G$. In \citep{Niculescu_Stroh_Zsido__noncmt_recur,Zsido__splitting_noncomm_dyn_sys},
a vector $\z\in\H$ is called \emph{almost periodic} if its orbit
$\left\{ U_{s}\z:s\in G\right\} $ is relatively compact in $\H$.
In this setting, however, there are other notions of almost periodicity
\citep{Jacobs_ergodic_thy_1956,deLeeuw_Glicksberg__appl_almst_per_com}.
We summarize the approaches to defining the set of almost periodic
vectors: {\renewcommand{\theenumi}{(\roman{enumi})}\renewcommand{\labelenumi}{\theenumi}
\begin{enumerate}
\item \label{enu:classical_AP_1}the closed linear span of the \emph{unitary
subspaces} (Jacobs \citep{Jacobs_ergodic_thy_1956}, de~Leeuw and
Glicksberg \citep[Definition of $B_p$, p.~75]{deLeeuw_Glicksberg__appl_almst_per_com});
\item \label{enu:classical_AP_2}the \emph{reversible} vectors (Jacobs \citep{Jacobs_ergodic_thy_1956},
de~Leeuw and Glicksberg \citep[Definition of $B_r$, p.~73]{deLeeuw_Glicksberg__appl_almst_per_com});
\item \label{enu:classical_AP_3}the vectors whose orbits are relatively
compact \citep{Niculescu_Stroh_Zsido__noncmt_recur,Zsido__splitting_noncomm_dyn_sys}
(see above).
\end{enumerate}
Approaches} \prettyref{enu:classical_AP_1} and \prettyref{enu:classical_AP_2}
are equivalent by \citep[Theorem 4.10]{deLeeuw_Glicksberg__appl_almst_per_com},
while their equivalence to \prettyref{enu:classical_AP_3} is a consequence
of a result similar to \citep[Lemma 4.1]{Niculescu_Stroh_Zsido__noncmt_recur}
together with the theory of unitary representations of compact groups.
Approaches \prettyref{enu:classical_AP_2} and \prettyref{enu:classical_AP_3},
which rely on the period of the vector, have the clear advantage of
being far more tangible than \prettyref{enu:classical_AP_1} when
testing for almost periodicity.

In the quantum setting, it is desirable to find a definition that
would be an obvious generalization of \prettyref{enu:classical_AP_3}.
Nevertheless, it seems that the only feasible definition in general
is a standard adaptation of \prettyref{enu:classical_AP_1}. In \prettyref{sub:CAP_motivation}
we give some motivation to almost periodicity in the quantum setting.
It is not essential for understanding the rest of the paper, but it
does put things in the right perspective. In \prettyref{sub:CAP_compactification}
we introduce our definition (\ref{def:CAP}) of complete almost periodicity
and prove the fundamental compactification result, \prettyref{thm:compactification_CAP}.
Then we show in \prettyref{cor:CAP_second_def} that under some assumptions,
it is indeed possible to give an equivalent definition of complete
almost periodicity, which is very close to \prettyref{enu:classical_AP_3}.

\subsection{\label{sub:CAP_motivation}Motivation}

The following easy lemma serves as some motivation for a possible
generalization of almost periodicity to the quantum setting.
\begin{lem}
\label{lem:classical_almost_per}In the classical setting of \citep{Niculescu_Stroh_Zsido__noncmt_recur,Zsido__splitting_noncomm_dyn_sys}
(see the introduction to this section), suppose that $\mu$ is a positive
Borel measure on $G$ that is finite on compact sets and nonzero on
open sets. For $\z\in\H$, define $T_{\z}\in B(L^{1}(G,\mu),\H)$
by $T_{\z}(\theta):=\int_{G}\theta(s)U_{s}\z\,\mathrm{d}s$, $\theta\in L^{1}(G,\mu)$.
Then $\z$ is almost periodic $\iff$ $T_{\z}$ is compact.\end{lem}
\begin{proof}
Let $\z\in\H$. By Mazur's theorem \citep[Theorem V.2.6]{DS1}, $\z$
is almost periodic if and only if $\cco\left\{ U_{s}\zeta:s\in G\right\} $
is compact. For every $\theta\in L^{1}(G)$ with $\theta\geq0$ and
$\left\Vert \theta\right\Vert _{1}=1$, we have $T_{\z}(\theta)\in\cco\left\{ U_{s}\z:s\in G\right\} $.
Moreover, $U_{s}\zeta\in\overline{\{T_{\z}(\theta):0\leq\theta\in\Lone G,\norm{\theta}_{1}=1\}}$
for all $s\in G$. Hence the assertion follows.
\end{proof}
As normally happens when one moves from the classical to the quantum
setting, Banach space (``commutative'') notions are replaced by
their operator space (``noncommutative'') counterparts.
\begin{defn}[\citep{Runde__CAP_func}]
Let $E,F$ be operator spaces and $T\in\CB(E,F)$. We say that $T$
is \emph{completely compact} if for every $\e>0$ there exists a finite-dimensional
subspace $F_{\e}$ of $F$ such that $\cbnorm{Q_{F_{\e}}T}<\e$, where
$Q_{F_{\e}}:F\to F/F_{\e}$ is the quotient map. The space of these
operators is denoted by $\CK(E,F)$.
\end{defn}
Complete compactness implies compactness, and the two notions agree
when $E$ has the maximal operator space structure (e.g., in the classical
setting). The space $\CK(E,F)$ contains the cb-closure of the finite-rank
operators from $E$ to $F$, that is, the space of all operators in
$\CB(E,F)$ that may be viewed as elements of the injective tensor
product $E^{*}\tensori F$ via its completely isometric embedding
in $\CB(E,F)$ \citep[Proposition 8.1.2]{Effros_Ruan__book}. These
two spaces actually coincide when $F$ is a dual operator space and
$E^{*},F$ are injective \citep[Proposition 1.6]{Runde__CAP_func}.
Returning to our setting, we consider the operators $T_{\z}\in\CB(\Lone{\G},\H_{c})$,
$\z\in\H$ (see \prettyref{sec:prelim_notation}). If $\G$ is a LCQG,
then co-amenability implies that $\Linfty{\G}=\Lone{\G}^{*}$ is injective
\citep{Bedos_Tuset_2003}. All the foregoing suggests that the quantum
version of almost periodicity of a vector $\z\in\H$ should be based
on $\tilde{\a}(\z)=T_{\z}$ belonging to $\H_{c}\tensori\Linfty{\G}$
($\hookrightarrow\H_{c}\tensorn\Linfty{\G}\hookrightarrow\CB(\Lone{\G},\H_{c})$).
\begin{lem}
\label{lem:CAP_injective_tensor_prod}Set $\K:=\left\{ \z\in\H:\tilde{\a}(\z)\in\H_{c}\tensori\Linfty{\G}\right\} $.
Then $\K$ is a closed subspace of $\H$, and we have $\tilde{\a}(\K)\subseteq\K_{c}\tensori\Linfty{\G}$
and $\tilde{\a}(\overline{\Gamma\Gamma^{-1}(\K)})\subseteq\overline{\Gamma\Gamma^{-1}(\K)}_{c}\tensori\Linfty{\G}$.\end{lem}
\begin{proof}
That $\K$ is a closed subspace of $\H$ follows easily from \prettyref{lem:alpha_tilde}.
To show that $\tilde{\a}(\K)\subseteq\K_{c}\tensori\Linfty{\G}$,
let $\z\in\K$. Since $\tilde{\a}(\zeta)=T_{\zeta}\in\H_{c}\tensori\Linfty{\G}$,
we need to prove that $\Img T_{\z}\subseteq\K$; indeed, this is enough
because of the existence of the (orthogonal, thus completely contractive)
projection from $\H$ onto $\K$. Fix $\vartheta\in\Lone{\G}$. For
every $\theta\in\Lone{\G}$ we have $T_{T_{\z}(\vartheta)}(\theta)=T_{\z}(\theta*\vartheta)$
(\prettyref{rem:alpha_tilde}), so letting $S:\Lone{\G}\to\Lone{\G}$
be given by $\theta\mapsto\theta*\vartheta$, we conclude that $T_{T_{\z}(\vartheta)}=T_{\zeta}S$.
Since $T_{\z}\in\H_{c}\tensori\Linfty{\G}$ and $S\in\CB(\Lone{\G})$,
we have $T_{T_{\z}(\vartheta)}\in\H_{c}\tensori\Linfty{\G}$, i.e.,
$T_{\z}(\vartheta)\in\K$. The last assertion follows similarly as
$T_{\Gamma(a)}(\theta)=\Gamma\bigl((\i\tensor\theta)\a(a)\bigr)$
for all $a\in N$ and $\theta\in\Lone{\G}$.\end{proof}
\begin{cor}
For every $a\in\Gamma^{-1}(\K)$ we have $\a(a)\in\Gamma^{-1}(\K)\tensorf\Linfty{\G}$,
where $\tensorf$ stands for the normal Fubini tensor product of dual
operator spaces.\end{cor}
\begin{lem}
\label{lem:E_tensor_H_approx}Let $E$ be an operator space and $\K$
be a Hilbert space. Let $x\in\K_{c}\tensori E$, and define $E_{x}:=\clinspan\left\{ (\rho^{*}\tensor\i)x:\rho\in\K\right\} $.
Then $x\in\K_{c}\tensori E_{x}$.
\end{lem}
The elementary proof is left to the reader.
\begin{prop}
\label{prop:comult_min_tensor}Set again $\K:=\left\{ \z\in\H:\tilde{\a}(\z)\in\H_{c}\tensori\Linfty{\G}\right\} $,
and let $\Aa$ be the unital $C^{*}$-subalgebra of $\Linfty{\G}$
generated by $\left\{ (\eta^{*}\tensor\i)\tilde{\a}(\z):\z\in\K,\eta\in\H\right\} $.
Then $\Delta$ restricts to a $C^{*}$-algebraic co-multiplication
$\Aa\to\Aa\tensormin\Aa$. In the case that $\G$ is a LCQG, $\Aa$
is a $C^{*}$-subalgebra of $M(C_{0}(\G))$.\end{prop}
\begin{proof}
It follows from Lemmas \ref{lem:CAP_injective_tensor_prod} and \ref{lem:E_tensor_H_approx}
that $\tilde{\a}(\K)\subseteq\K_{c}\tensori\Aa$. Using the formula
$(\i\tensor\Delta)\tilde{\a}=(\tilde{\a}\tensor\i)\tilde{\a}$ of
\prettyref{lem:alpha_tilde} we have for $\z\in\K,\eta\in\H$,
\begin{equation}
\Delta((\eta^{*}\tensor\i)\tilde{\a}(\z))=(\eta^{*}\tensor\i\tensor\i)(\i\tensor\Delta)\tilde{\a}(\z)=(\eta^{*}\tensor\i\tensor\i)(\tilde{\a}\tensor\i)\tilde{\a}(\z)\in\Aa\tensormin\Aa,\label{eq:compactification_CAP}
\end{equation}
since for $C^{*}$-algebras, the minimal and the operator space injective
tensor products coincide. Consequently, $\Delta(\Aa)\subseteq\Aa\tensormin\Aa$.
In the case that $\G$ is a LCQG, we have $U\in M(\mathbb{K}(\H)\tensormin C_{0}(\G))$
because $U$ is a co-representation of $\G$ (this is folklore; see,
e.g., \citep{Brannan_Daws_Samei__cb_rep_of_conv_alg_of_LCQGs}, Corollary
4.12 and the proof of Theorem 4.9), and since $(\eta^{*}\tensor\i)\tilde{\a}(\z)=(\om_{\z,\eta}\tensor\i)(U)$
for all $\z,\eta\in\H$, we obtain that $\Aa\subseteq M(C_{0}(\G))$.
\end{proof}
The pair $(\Aa,\Delta|_{\Aa})$ is, therefore, a compact quantum semigroup
(which is a compactification of $\G$), but normally not a compact
quantum group.

\subsection{\label{sub:CAP_compactification}The compactification}

The previous discussion indicates that \emph{optimally}, we would
say that $\z\in\H$ is almost periodic in the quantum setting if $\tilde{\a}(\z)\in\H_{c}\tensori\Linfty{\G}$.
Nonetheless, this definition is too weak for the development of the
rest of the theory, and we now introduce a more restrictive one, as
follows. It is the noncommutative version of approach \prettyref{enu:classical_AP_1}
above (also compare So{\ldash}tan \citep{Soltan__quantum_Bohr_comp}
and Woronowicz \citep{Woronowicz__remark_on_CQGs}).
\begin{defn}
\label{def:CAP}The set $\H^{\mathrm{CP}}$ of \emph{completely periodic}
\emph{vectors} consists of all $\z\in\H$ with the following property:
there exists a finite-dimensional sub-representation $u$ of $U$
on a space that contains $\z$, such that both $u$ and $u^{\mathrm{t}}$
are invertible.%
\footnote{Here $u^{\mathrm{t}}$ is the transpose of $u$, with respect to some
basis of the space.%
}

The \emph{completely almost periodic} \emph{vectors} are the elements
of $\H^{\CAP}:=\clinspan\H^{\mathrm{CP}}$, and the \emph{completely
almost periodic operators} are the elements of $N^{\CAP}:=\Gamma^{-1}(\H^{\CAP})$.
\end{defn}
We will give in \prettyref{cor:CAP_second_def} another characterization
of $\H^{\CAP}$ under additional assumptions.
\begin{rem}
The set $N^{\CAP}$ is a weakly closed subspace of $N$.
\end{rem}
The main objective of this paper is to address the following questions:{\renewcommand{\theenumi}{\Alph{enumi}}\renewcommand{\labelenumi}{(\theenumi)}
\begin{enumerate}
\item \label{enu:question_1}Is $N^{\CAP}$ a von Neumann algebra?
\item \label{enu:question_2}Is $\Gamma(N^{\CAP})$ dense in $\H^{\CAP}$?
\item \label{enu:question_3}Is $N^{\CAP}$ globally invariant under the
modular automorphism group $\sigma^{\om}$ of $\om$?
\end{enumerate}
}
\begin{defn}
Let $\G$ be a Hopf--von Neumann algebra. A compact quantum group
$\HH=(C(\HH),\Delta_{\HH})$ is a \emph{compactification} of $\G$
if $C(\HH)\subseteq\Linfty{\G}$ and $\Delta_{\HH}=\Delta|_{C(\HH)}$.
If $\G$ is a LCQG, we further require that $C(\HH)\subseteq M(C_{0}(\G))$.\end{defn}
\begin{rem}
If $\G$ is co-amenable, then so is $\HH$, because if $\left(\epsilon_{\l}\right)$
is a bounded left approximate identity of $\G$, then any cluster
point of $\left(\epsilon_{\l}|_{C(\HH)}\right)$ in $C(\HH)^{*}$
is a co-unit of $\HH$ (cf.~B\'edos and Tuset \citep[Theorem 3.1]{Bedos_Tuset_2003}). \end{rem}
\begin{thm}
\label{thm:compactification_CAP}There exists a co-amenable compactification
$\HH=(C(\HH),\Delta_{\HH})$ of $\G$ such that $U|_{\H^{\CAP}\tensor L^{2}(\G)}$
is a unitary co-representation of $\HH$ on $\H^{\CAP}$ in the $C^{*}$-algebraic
sense. To elaborate, $C(\HH)$ is the unital $C^{*}$-subalgebra of
$\Linfty{\G}$ generated by $\left\{ (\om_{\z,\eta}\tensor\i)(U):\z\in\H^{\CAP},\eta\in\H\right\} $
and $\Delta_{\HH}$ is the restriction of $\Delta$ to $C(\HH)$.\end{thm}
\begin{proof}
From \prettyref{def:CAP} it follows that $\Delta(C(\HH))\subseteq C(\HH)\tensormin C(\HH)$
and that $U$ is invariant under $\H^{\CAP}\tensor L^{2}(\G)$. Taking
$\z\in\H^{\mathrm{CP}}$ and letting $u$ be as in \prettyref{def:CAP},
the elements $(\om_{\z,\eta}\tensor\i)(U)=(\eta^{*}\tensor\i)\tilde{\a}(\z)$,
$\eta\in\H$, are just linear combinations of the matrix elements
of $u$. Since the co-representation $u$ and its transpose $u^{\mathrm{t}}$
are invertible and the unital $C^{*}$-algebra $C(\HH)$ is generated
by $\left\{ (\om_{\z,\eta}\tensor\i)(U):\z\in\H^{\mathrm{CP}},\eta\in\H\right\} $,
$\HH$ is a compact quantum group by Maes and Van Daele \citep[Proposition 3.8]{Maes_van_Daele__notes_CQGs}
(see also Woronowicz \citep{Woronowicz__remark_on_CQGs}). If $\G$
is a LCQG, then $C(\HH)\subseteq M(C_{0}(\G))$ by \prettyref{prop:comult_min_tensor}
as $\H^{\CAP}$ is contained in the subspace $\K$ therein.

The restriction $V:=U|_{\H^{\CAP}\tensor L^{2}(\G)}$ is now a co-representation
of $\G$ on $\H^{\CAP}$. By definition, $V$ (which is isometric)
has dense range, so it is unitary. So we need only establish that
$V\in M(\mathbb{K}(\H^{\CAP})\tensormin C(\HH))$. Let $\z\in\H^{\mathrm{CP}}$,
$\eta\in\H^{\CAP}$ and $x\in C(\HH)$. Since $\tilde{\a}(\H^{\mathrm{CP}})\subseteq\H^{\mathrm{CP}}\odot C(\HH)$,
we get $V\left((\z\tensor\eta^{*})\tensor x\right)\in\mathbb{K}(\H^{\CAP})\odot C(\HH)$.
Let $\H_{1}$ be the finite-dimensional subspace associated with $\z$
in \prettyref{def:CAP} and $p_{\H_{1}}$ the projection of $\H^{\CAP}$
onto $\H_{1}$. Since $V$ is invariant under $\H_{1}\tensor L^{2}(\G)$
and $u:=V|_{\H_{1}\tensor L^{2}(\G)}\in B(\H_{1})\odot C(\HH)$ is
unitary, we have $\left((\eta\tensor\z^{*})\tensor x\right)V=\left((\eta\tensor\z^{*})\tensor x\right)u(p_{\H_{1}}\tensor\one)\in\mathbb{K}(\H^{\CAP})\odot C(\HH)$.
The foregoing implies that $V(k\tensor x)$ and $(k\tensor x)V$ belong
to $\mathbb{K}(\H^{\CAP})\tensormin C(\HH)$ for every $k\in\mathbb{K}(\H^{\CAP})$
and $x\in C(\HH)$, hence $V\in M(\mathbb{K}(\H^{\CAP})\tensormin C(\HH))$
as desired.\end{proof}
\begin{cor}
\label{cor:m_faithful_on_H}Every left- (or right-) invariant state
of $\G$ is faithful on $C(\HH)$. \end{cor}
\begin{proof}
Let $m$ be a left-invariant state of $\G$. Since $m|_{C(\HH)}$
is a left-invariant state of the compact quantum group $\HH$, it
must be equal to the Haar state of $\HH$ \citep{Woronowicz__symetries_quantiques,Maes_van_Daele__notes_CQGs},
which is faithful as $\HH$ is co-amenable \citep[Theorem 2.2]{Bedos_Murphy_Tuset__coam_of_CQGs}.
\end{proof}
Under mild assumptions, we can provide another characterization of
$\H^{\CAP}$, one which bears a stronger resemblance to the definition
of almost periodicity used in \citep{Niculescu_Stroh_Zsido__noncmt_recur,Zsido__splitting_noncomm_dyn_sys}
for the classical case (approach \prettyref{enu:classical_AP_3} above).
\begin{defn}
We say that a quantum semigroup $\G$ satisfies condition (I) if every
isometric or co-isometric finite-dimensional co-representation of
$\G$ is unitary.
\end{defn}
Condition (I) is satisfied when $\G$ is a LCQG, in which case every
isometric or co-isometric co-representation of $\G$ is unitary by
\citep[Corollaries 4.11, 4.12]{Brannan_Daws_Samei__cb_rep_of_conv_alg_of_LCQGs}.
It is also satisfied when $\Linfty{\G}$ is a finite von Neumann algebra.
\begin{cor}
\label{cor:CAP_second_def}Let $\H^{\mathrm{CP2}}$ be the set of
all vectors $\z\in\H$ satisfying the following conditions:
\begin{enumerate}
\item \label{enu:CAP2_1}there exists a finite-dimensional sub-representation
of $U$ on a space that contains $\z$;%
\footnote{this is evidently equivalent to $T_{\z}:\Lone{\G}\to\H$ having finite
rank%
} and
\item \label{enu:CAP2_2}for every $\eta\in\H$ with $x:=(\om_{\z,\eta}\tensor\i)(U)\neq0$,
we have $m(x^{*}x)>0$ for some left-invariant mean $m$ on $\G$.
\end{enumerate}
We have $\H^{\mathrm{CP}}\subseteq\H^{\mathrm{CP2}}$. If $\G$ satisfies
condition (I), then $\H^{\mathrm{CP}}=\H^{\mathrm{CP2}}$, so that
$\clinspan\H^{\mathrm{CP2}}=\H^{\CAP}$.
\end{cor}
Condition \prettyref{enu:CAP2_2} originates in a concrete interpretation
of periodicity. In the setting of Niculescu, Str\"{o}h and Zsid\'{o}
\citep{Niculescu_Stroh_Zsido__noncmt_recur}, one has the following:
a vector $\z\in\H$ is almost periodic if and only if for every $\e>0$,
the set $\left\{ n\in\Z_{+}:\left\Vert U^{n}\z-\z\right\Vert <\e\right\} $
is relatively dense in $\Z_{+}$ \citep[Corollary 9.10]{Niculescu_Stroh_Zsido__noncmt_recur}.
This obviously implies that if $\z\in\H$ is almost periodic and $\eta\in\H$
is such that the function $x:\Z_{+}\to\C$, $n\mapsto\left\langle U^{n}\z,\eta\right\rangle $,
is not identically zero, then $m(\left|x\right|^{2})>0$ for \emph{every}
invariant mean $m$ on $\Z_{+}$. A similar assertion can be stated
in the more general setting of Zsid\'{o} \citep{Zsido__splitting_noncomm_dyn_sys}
as well. Thus, \prettyref{enu:CAP2_2} can be viewed as a weak type
of recurrence (which is automatic in the classical setting).
\begin{proof}[Proof of \prettyref{cor:CAP_second_def}]
If $\z\in\H^{\mathrm{CP}}$ and $\eta\in\H$ are such that $x:=(\om_{\z,\eta}\tensor\i)(U)\neq0$,
then $m(x^{*}x)>0$ for\emph{ every }left- (or right-) invariant mean
$m$ on $\G$ by \prettyref{cor:m_faithful_on_H} as $x\in C(\HH)$.
In conclusion, $\H^{\mathrm{CP}}\subseteq\H^{\mathrm{CP2}}$.

Let $\z\in\H^{\mathrm{CP2}}$ with $\norm{\z}=1$. Write $\tilde{\a}(\z)=\z_{1}\tensor a_{1}+\ldots+\z_{n}\tensor a_{n}$,
where $a_{1},\ldots,a_{n}$ are linearly independent and $\z_{1},\ldots,\z_{n}$
are orthonormal. Let $\H_{1}:=\linspan\left\{ \z_{1},\ldots,\z_{n}\right\} $.
By co-amenability, $\z\in\H_{1}$ (see \prettyref{lem:alpha_tilde}),
so we may assume that $\z=\z_{1}$.

For each $1\leq i\leq n$, fix $\om_{i}\in\Lone{\G}$ with $\om_{i}(a_{j})=\delta_{ij}$
for all $1\leq j\leq n$. Now $\tilde{\a}(\z_{i})=\tilde{\a}((\i\tensor\om_{i})\tilde{\a}(\z))=(\i\tensor\i\tensor\om_{i})(\tilde{\a}\tensor\i)\tilde{\a}(\z)=(\i\tensor(\i\tensor\om_{i})\Delta)\tilde{\a}(\z)$.
Hence $\H_{1}$ is invariant under $U$. Let $u\in B(\H_{1})\odot\Linfty{\G}$
be the co-representation of $\G$ on $\H_{1}$ given by $u:=U|_{\H_{1}\tensor L^{2}(\G)}$.
Since $u$ is an isometry and $\G$ satisfies condition (I), $u$
is unitary. Write $u$ as a matrix $\left(u_{ij}\right)_{i,j=1}^{n}\in M_{n}(\Linfty{\G})$
by setting $u_{ij}:=(\z_{i}^{*}\tensor\i)\tilde{\a}(\z_{j})$. In
particular, $a_{i}=u_{i1}$ for every $i$.

We need only show that $u^{\mathrm{t}}$ is invertible to establish
that $\z\in\H^{\mathrm{CP}}$. Let $m$ be a left-invariant mean on
$\G$ as in the corollary's statement. Since $u$ is a co-representation
of $\G$,
\[
m(a_{i}^{*}a_{j})\one=(\i\tensor m)\Delta(a_{i}^{*}a_{j})=\sum_{k,l=1}^{n}u_{ik}^{*}u_{jl}m(a_{k}^{*}a_{l})\qquad(\forall1\leq i,j\leq n).
\]
That is, letting $g:=\left(m(a_{k}^{*}a_{l})\right)_{k,l}$, we get
$u^{\mathrm{t}*}(g\tensor\one)u^{\mathrm{t}}=g\tensor\one$. The positive
semi-definite matrix $g$ is invertible, for otherwise there is $0\neq c=\left(c_{i}\right)\in\C^{n}$
such that $gc=0$. Thus, writing $x:=\sum_{k=1}^{n}c_{k}a_{k}=((\sum_{k=1}^{n}\overline{c_{k}}\z_{k})^{*}\tensor\i)\tilde{\a}(\z)$,
we have $x\neq0$ (because $a_{1},\ldots,a_{n}$ are linearly independent)
but $m(x^{*}x)=0$, a contradiction. Consequently, $(g^{-1/2}\tensor\one)u^{\mathrm{t}*}(g^{1/2}\tensor\one)$
is a co-isometric (finite-dimensional) co-representation of $\G$.
Hence, by condition (I), it is invertible. Therefore, so is $u^{\mathrm{t}}$.
This completes the proof.
\end{proof}

\section{\label{sec:analysis_of_CAP}Analysis of $\H^{\CAP}$ and $N^{\CAP}$}

In this section we settle affirmatively Questions \ref{enu:question_1}--\ref{enu:question_3}
(the last one under additional conditions). The two main tools are
the compactification result of \prettyref{sec:CAP} and modular theory.
Some of the results and techniques should be compared to Boca \citep{Boca__ergodic_actions}.

Fix a complete family $\left(u^{\gamma}\right)_{\gamma\in\mathrm{Irred}(\HH)}$
of irreducible unitary co-representations of $\HH$ (see \citep{Woronowicz__symetries_quantiques}
for details). For each $\gamma\in\mathrm{Irred}(\HH)$, $u^{\gamma}=(u_{ij}^{\gamma})_{i,j=1}^{n(\gamma)}$
belongs to $M_{n(\gamma)}\odot C(\HH)$. As $\HH$ is a compactification
of $\G$, every such $u^{\gamma}$ is also a co-representation of
$\G$. Consider the construction in \prettyref{thm:genr_mean_ergodic}
with $C$ being $M_{n(\gamma)}$ and $\mu$ being the normalized trace
$\tr_{\gamma}$ on $M_{n(\gamma)}$. Denoting the GNS construction
of $(M_{n(\gamma)},\tr_{\gamma})$ by $(\H_{\gamma},\Gamma_{\gamma})$,
we get the associated projection $P_{u^{\gamma}}\in B(\H_{\gamma}\tensor\H)$
and normal projection $E_{u^{\gamma}}:M_{n(\gamma)}\odot N\to N_{u^{\gamma}}=\left\{ b\in M_{n(\gamma)}\odot N:(\i\tensor\a)(b)=u_{13}^{\gamma*}b_{12}\right\} $.
Write $(e_{ij}^{\gamma})$ for a system of matrix units for $M_{n(\gamma)}$.
\begin{lem}
\label{lem:CAP_irred_H_P_E}We have $\Img P_{u^{\gamma}}\subseteq\H_{\mu}\tensor\H^{\CAP}$
and $N_{u^{\gamma}}=\Img E_{u^{\gamma}}\subseteq M_{n(\gamma)}\odot N^{\CAP}$
for all $\gamma\in\mathrm{Irred}(\HH)$.\end{lem}
\begin{proof}
Write $u,n$ for $u^{\gamma},n(\gamma)$, respectively. Let $b\in\Img E_{u}$.
For all $1\leq i,j\leq n$, we have $\a(b_{ij})=\sum_{k=1}^{n}b_{kj}\tensor(u^{\gamma*})_{ik}$
, hence $\tilde{\a}(\Gamma(b_{ij}))=\sum_{k=1}^{n}\Gamma(b_{kj})\tensor(u^{\gamma*})_{ik}$.
Fix $j$. The operators $b_{ij}$, $1\leq i\leq n$, are linearly
independent unless all zero, as $u$ is irreducible. Since $u^{\mathrm{t}}$
is invertible \citep{Woronowicz__symetries_quantiques,Maes_van_Daele__notes_CQGs},
the sub-representation of $U$ on $\linspan\left\{ \Gamma(b_{ij})\right\} _{i=1}^{n}$
satisfies the condition of \prettyref{def:CAP}, so that $\Gamma(b_{ij})\in\H^{\mathrm{CP}}$,
thus $b_{ij}\in N^{\CAP}$, for every $i$. This proves the second
assertion. The first one follows by \prettyref{thm:genr_mean_ergodic},
\prettyref{enu:genr_mean_ergodic__2}.\end{proof}
\begin{lem}
\label{lem:alpha_of_N_cap}For every $a\in N^{\CAP}$ and $\rho\in N_{*}$
we have $(\rho\tensor\i)\a(a)\in C(\HH)$.\end{lem}
\begin{proof}
If $\rho$ is of the form $\rho(x)=\left\langle \Gamma(x),\eta\right\rangle $
for some $\eta\in\H$, then $(\rho\tensor\i)\a(a)=(\eta^{*}\tensor\i)\tilde{\a}(\Gamma(a))\in C(\HH)$
by construction. The Hahn--Banach theorem implies that the subspace
of functionals of this form is norm dense in $N_{*}$, and the assertion
follows.
\end{proof}
Let now $A_{u^{\gamma}}$ ($\subseteq N^{\CAP}$) denote the ``right
leg of $N_{u^{\gamma}}=\Img E_{u^{\gamma}}$'', namely the span of
all matrix elements of matrices in $\Img E_{u^{\gamma}}$. Similarly,
let $B_{u^{\gamma}}$ ($\subseteq\H^{\CAP}$) denote the ``right
leg of $\Img P_{u^{\gamma}}$''.
\begin{prop}
\label{prop:CAP_densities}The following assertions hold.
\begin{enumerate}
\item \label{enu:CAP_densities_1}The set $\bigcup_{\gamma\in\mathrm{Irred}(\HH)}B_{u^{\gamma}}$
is total in $\H^{\CAP}$.
\item \label{enu:CAP_densities_2}The set $\Gamma\left(\bigcup_{\gamma\in\mathrm{Irred}(\HH)}A_{u^{\gamma}}\right)$
is total in $\H^{\CAP}$.
\item \label{enu:CAP_densities_3}$\overline{\Gamma(N^{\CAP})}=\H^{\CAP}$
(Question \ref{enu:question_2}).
\item \label{enu:CAP_densities_4}The set $\bigcup_{\gamma\in\mathrm{Irred}(\HH)}A_{u^{\gamma}}$
is weakly total in $N^{\CAP}$.
\end{enumerate}
\end{prop}
\begin{proof}
Let $\eta\in\H$ be orthogonal to $\bigcup_{\gamma\in\mathrm{Irred}(\HH)}B_{u^{\gamma}}$.
Let $m$ be an invariant mean on $\G$, and fix $\left(m_{\kappa}\right)$
as in \prettyref{rem:amen_norm}. For all $\z\in\H^{\CAP}$, $\gamma\in\mathrm{Irred}(\HH)$
and $1\leq i,j\leq n(\gamma)$ we have, by \prettyref{thm:genr_mean_ergodic},
\[
0=\eta^{*}\left((\Gamma_{\gamma}(e_{ij}^{\gamma})^{*}\tensor\i)P_{u^{\gamma}}(\Gamma_{\gamma}(\one)\tensor\z)\right)=\lim_{\l}m_{\l}(u_{ij}^{\gamma}\cdot(\eta^{*}\tensor\i)\tilde{\alpha}(\z))=m(u_{ij}^{\gamma}\cdot(\eta^{*}\tensor\i)\tilde{\alpha}(\z)).
\]
Hence $x:=(\eta^{*}\tensor\i)\tilde{\alpha}(\z)$ belongs to $C(\HH)$
and satisfies $m(u_{ij}^{\gamma}x)=0$. Since the span of all $u_{ij}^{\gamma}$
for $\gamma\in\mathrm{Irred}(\HH)$, $1\leq i,j\leq n(\gamma)$, is
dense in $C(\HH)$, we have $m(x^{*}x)=0$ which, by \prettyref{cor:m_faithful_on_H},
implies that $x=0$. By co-amenability of $\G$, we get
\[
0=\lim_{\l}\epsilon_{\l}\left[(\eta^{*}\tensor\i)\tilde{\alpha}(\z)\right]=\lim_{\l}\left\langle (\i\tensor\epsilon_{\l})\tilde{\alpha}(\z),\eta\right\rangle =\left\langle \zeta,\eta\right\rangle
\]
for all $\z\in\H^{\CAP}$ (see \prettyref{lem:alpha_tilde}). This
proves \prettyref{enu:CAP_densities_1} by \prettyref{lem:CAP_irred_H_P_E}.

From \prettyref{thm:genr_mean_ergodic}, \prettyref{enu:genr_mean_ergodic__2}
we know that $\overline{\Gamma(A_{u^{\gamma}})}=B_{u^{\gamma}}$ for
all $\gamma\in\mathrm{Irred}(\HH)$, so \prettyref{enu:CAP_densities_2}
follows. This in turn implies \prettyref{enu:CAP_densities_3} because
$\bigcup_{\gamma\in\mathrm{Irred}(\HH)}A_{u^{\gamma}}\subseteq N^{\CAP}$
(\prettyref{lem:CAP_irred_H_P_E}).

The proof of \prettyref{enu:CAP_densities_4} is very similar to that
of \prettyref{enu:CAP_densities_1}. Suppose that $\rho\in N_{*}$
vanishes on $A_{u^{\gamma}}$ for every $\gamma\in\mathrm{Irred}(\HH)$.
This means that for all $a\in N^{\CAP}$, $\gamma\in\mathrm{Irred}(\HH)$
and $1\leq i,j\leq n(\gamma)$, we have $m(u_{ij}^{\gamma}\cdot(\rho\tensor\i)\a(a))=0$.
From \prettyref{lem:alpha_of_N_cap} we obtain $x:=(\rho\tensor\i)\a(a)\in C(\HH)$.
As above, we infer that $x=0$, and thus $\rho(a)=0$ by co-amenability.
\end{proof}
Henceforth we denote the (faithful) Haar state of $C(\HH)$ by $h$.
The notations $\sigma^{h}$ and $\sigma^{\om}$ stand for the modular
automorphism groups of $h$ and $\omega$, respectively. We write
$\overline{u}$ for the contragradient to a unitary irreducible co-representation
$u$ of $\HH$, that is, the unique element of $\left(u^{\gamma}\right)_{\gamma\in\mathrm{Irred}(\HH)}$
that is equivalent to $u^{*\mathrm{t}}$. (Note that $u^{*\mathrm{t}}$
is denoted by $\overline{u}$ in \citep{Maes_van_Daele__notes_CQGs}.)
\begin{prop}
\label{prop:A_u}For every $\gamma,\delta\in\mathrm{Irred}(\HH)$
we have
\begin{enumerate}
\item \label{enu:A_u__1}$A_{u^{\gamma}}^{*}=A_{\overline{u^{\gamma}}}$
;
\item \label{enu:A_u__2}$A_{u^{\gamma}}A_{u^{\delta}}\subseteq\linspan_{\be\in\mathrm{Irred}(\HH)}A_{u^{\be}}$.
\end{enumerate}
\end{prop}
\begin{proof}
\prettyref{enu:A_u__1} Write $u,n$ for $u^{\gamma},n(\gamma)$,
respectively, and let $a=(a_{ij})_{i,j=1}^{n}\in M_{n}\odot N$. Since
$E_{u}$ is a projection, we may assume, for the purpose of showing
that the matrix elements of $E_{u}(a)^{*}$ belong to $A_{\overline{u}}$,
that $a\in\Img E_{u}$. Therefore $\a(a_{ij})\in N\odot C(\HH)$ for
every $i,j$ (Lemmas \ref{lem:CAP_irred_H_P_E} and \ref{lem:alpha_of_N_cap}).
Let $m$ be an invariant mean on $\G$. Recalling that $h=m|_{C(\HH)}$
(see proof of \prettyref{cor:m_faithful_on_H}), we have, by \prettyref{thm:genr_mean_ergodic}
and \prettyref{rem:amen_norm},
\[
E_{u}(a)^{*}=\left\{ (\i\tensor\i\tensor h)\left[u_{13}\cdot(\i\tensor\a)(a)\right]\right\} ^{*}=(\i\tensor\i\tensor h)\left[(\i\tensor\a)(a^{*})\cdot u_{13}^{*}\right].
\]
Let $E,F\in GL_{n}$, $F$ positive definite, satisfy $\overline{u}=(E\tensor\one)u^{*\mathrm{t}}(E^{-1}\tensor\one)$
and $(\i\tensor\sigma_{i}^{h})\overline{u}=(F\tensor\one)\overline{u}(F\tensor\one)$
(see \citep{Woronowicz__symetries_quantiques}). Then $(\i\tensor\sigma_{i}^{h})(u^{*})=\left[(E^{-1}\tensor\one)(\i\tensor\sigma_{i}^{h})(\overline{u})(E\tensor\one)\right]^{\mathrm{t}}$,
and therefore
\begin{equation}
\begin{split}E_{u}(a)^{*} & =\sum_{i,j,k=1}^{n}e_{ij}\tensor(\i\tensor h)(\a(a_{ki}^{*})(\one\tensor u_{jk}^{*}))\\
 & =\sum_{i,j,k=1}^{n}e_{ij}\tensor(\i\tensor h)\left\{ \left[\one\tensor\left((E^{-1}F\tensor\one)\overline{u}(FE\tensor\one)\right)_{jk}\right]\a(a_{ki}^{*})\right\}
\end{split}
\label{eq:invariance_under_mod_aut_gr__E_u_a_star_1}
\end{equation}
by the Tomita--Takesaki theory. Denoting $G_{1}:=E^{-1}F$ and $G_{2}:=FE$,
we get
\begin{equation}
\begin{split}E_{u}(a)^{*\mathrm{t}} & =(\i\tensor\i\tensor h)\left((G_{1}\tensor\one\tensor\one)\overline{u}_{13}(G_{2}\tensor\one\tensor\one)(\i\tensor\a)(a^{*\mathrm{t}})\right)\\
 & =(G_{1}\tensor\one)(\i\tensor\i\tensor h)\left[\overline{u}_{13}(\i\tensor\a)\left((G_{2}\tensor\one)a^{*\mathrm{t}}\right)\right]\\
 & =(G_{1}\tensor\one)E_{\overline{u}}\left((G_{2}\tensor\one)a^{*\mathrm{t}}\right).
\end{split}
\label{eq:invariance_under_mod_aut_gr__E_u_a_star_2}
\end{equation}
This implies that $A_{u}^{*}\subseteq A_{\overline{u}}$. The converse
inclusion is obtained by interchanging the roles of $u$ and $\overline{u}$.

\prettyref{enu:A_u__2} Given $a\in\Img E_{u^{\gamma}}$ and $b\in\Img E_{u^{\delta}}$,
that is, $(\i\tensor\a)(a)=u_{13}^{\gamma*}a_{12}$ and $(\i\tensor\a)(b)=u_{13}^{\delta*}b_{12}$,
we get
\[
(\i_{M_{n(\gamma)}}\tensor\i_{M_{n(\delta)}}\tensor\a)(a_{13}b_{23})=u_{14}^{\gamma*}a_{13}u_{24}^{\delta*}b_{23}=u_{14}^{\gamma*}u_{24}^{\delta*}a_{13}b_{23}
\]
Thus, for $c:=a_{13}b_{23}\in M_{n(\gamma)}\odot M_{n(\delta)}\odot N$
and the unitary co-representation $u:=u_{23}^{\delta}u_{13}^{\gamma}$
of $\HH$, we have $c\in\Img E_{u}$. Since $u$ decomposes (with
respect to some basis) as $u=u^{\mathbf{\be}_{1}}\oplus\ldots\oplus u^{\mathbf{\be}_{n}}$
for suitable $\be_{1},\ldots,\be_{n}\in\mathrm{Irred}(\HH)$ \citep{Woronowicz__symetries_quantiques},
the desired conclusion follows.\end{proof}
\begin{thm}
\label{thm:N_CAP_vN_alg}The space $N^{\CAP}$ is a von Neumann subalgebra
of $N$ (Question \ref{enu:question_1}).\end{thm}
\begin{proof}
The subspace $\linspan\bigcup_{\gamma\in\mathrm{Irred}(\HH)}A_{u^{\gamma}}$
is a $*$-subalgebra of $N^{\CAP}$ by \prettyref{prop:A_u}. Thus
the assertion follows from \prettyref{prop:CAP_densities}, \prettyref{enu:CAP_densities_4}.
\end{proof}
We now address Question \ref{enu:question_3}, providing a few circumstances
under which it has an affirmative answer.%

\begin{thm}
\label{thm:invariance_under_mod_aut_gr_LCQGs}The von Neumann algebra
$N^{\CAP}$ is globally invariant under $\sigma^{\om}$ if $\G$ is
a LCQG.\end{thm}
\begin{proof}
Denote by $\tau^{\G},\tau^{\HH}$ the scaling groups of $\G,\HH$,
respectively. Since $\a$ is an $\om$-invariant action of $\G$ on
$N$, we get, by \prettyref{thm:action_with_invariant_state_LCQG},
\[
(\sigma_{t}^{\om}\tensor\tau_{-t}^{\G})\circ\a=\a\circ\sigma_{t}^{\om}\qquad(\forall t\in\R).
\]
Fix $\gamma\in\mathrm{Irred}(\HH)$ and denote $u:=u^{\gamma}$. There
exists a positive definite $F\in GL_{n(\gamma)}$ such that $(\i\tensor\tau_{t}^{\HH})u=(F^{it}\tensor\one)u(F^{-it}\tensor\one)$
for each $t\in\R$ \citep{Woronowicz__symetries_quantiques}. Thus
\citep[Proposition 5.45]{Kustermans_Vaes__LCQG_C_star} implies that
$(\i\tensor\tau_{t}^{\G})u=(F^{it}\tensor\one)u(F^{-it}\tensor\one)$.
If $a\in\Img E_{u}$, namely $(\i\tensor\a)(a)=u_{13}^{*}a_{12}$,
then for every $t\in\R$, we obtain
\[
\begin{split}(\i\tensor\a\circ\sigma_{t}^{\om})(a) & =\left(\i\tensor(\sigma_{t}^{\om}\tensor\tau_{-t}^{\G})\circ\a\right)(a)=(\i\tensor\sigma_{t}^{\om}\tensor\tau_{-t}^{\G})(u_{13}^{*}a_{12})\\
 & =(F^{-it}\tensor\one\tensor\one)u_{13}^{*}(F^{it}\tensor\one\tensor\one)\left((\i\tensor\sigma_{t}^{\om})(a)\right)_{12}.
\end{split}
\]
Therefore
\[
(\i\tensor\a)\left[(\i\tensor\sigma_{t}^{\om})\left((F^{it}\tensor\one)a\right)\right]=u_{13}^{*}\left[(\i\tensor\sigma_{t}^{\om})\left((F^{it}\tensor\one)a\right)\right]_{12},
\]
proving that $(\i\tensor\sigma_{t}^{\om})\left((F^{it}\tensor\one)a\right)\in\Img E_{u}$.
Thus $A_{u}$ is globally invariant under $\sigma^{\om}$. The result
follows from \prettyref{prop:CAP_densities}, \prettyref{enu:CAP_densities_4}.\end{proof}
\begin{thm}
\label{thm:invariance_under_mod_aut_gr}The von Neumann algebra $N^{\CAP}$
is globally invariant under $\sigma^{\om}$ if the following is true:
for every $\gamma\in\mathrm{Irred}(\HH)$ there exists an invariant
mean $m$ on $\G$ such that for each $1\leq i,j\leq n(\gamma)$ there
is $v\in C(\HH)$ with $m(xu_{ij}^{\gamma})=m(vx)$ for all $x\in\Linfty{\G}$.
\end{thm}
Of course, this condition holds trivially if $\G$ is classical, namely
$\Linfty{\G}$ is commutative.
\begin{proof}
First, note that $v=\sigma_{i}^{h}(u_{ij}^{\gamma})$ in the theorem's
condition by \prettyref{cor:m_faithful_on_H}. We use the notation
and results of \prettyref{prop:A_u}. Fix $\gamma\in\mathrm{Irred}(\HH)$,
and write $u,n$ for $u^{\gamma},n(\gamma)$, respectively. Let $m$
be the invariant mean suitable for $\overline{u}$ as above. In the
proof of \prettyref{prop:A_u}, \prettyref{enu:A_u__1}, we made the
assumption that $a\in\Img E_{u}$. This cannot be done in the present
proof. Taking an arbitrary $a=(a_{ij})_{i,j=1}^{n}\in M_{n}\odot N$,
we repeat the calculations of \prettyref{prop:A_u} with $m$ replacing
$h$. The maneuver of \prettyref{eq:invariance_under_mod_aut_gr__E_u_a_star_1}
is now justified by the assumption on $m$, and we infer from \prettyref{eq:invariance_under_mod_aut_gr__E_u_a_star_2}
that $E_{u}(a)^{*\mathrm{t}}=(G_{1}\tensor\one)E_{\overline{u}}\left((G_{2}\tensor\one)a^{*\mathrm{t}}\right)$.

Let $J,\nabla$ denote the modular conjugation and modular operator,
respectively, associated with $(N,\om)$. That is, the closure $S$
of the conjugate-linear map $\Gamma(x)\mapsto\Gamma(x^{*})$, $x\in N$,
over $\H$, has polar decomposition $S=J\nabla^{1/2}$. Write $A$
for the componentwise complex conjugation map $\Gamma_{\gamma}\left((a_{ij})\right)\mapsto\Gamma_{\gamma}\left((\overline{a_{ij}})\right)$
over $\H_{\gamma}$. Then by the previous paragraph and \prettyref{thm:genr_mean_ergodic},
\[
\begin{split}(A\tensor S)P_{u}\Xi(a) & =(A\tensor S)\Xi(E_{u}(a))=\Xi(E_{u}(a)^{*\mathrm{t}})\\
 & =\Xi\left[(G_{1}\tensor\one)E_{\overline{u}}\left((G_{2}\tensor\one)a^{*\mathrm{t}}\right)\right]=(G_{1}\tensor\one)\Xi\left[E_{\overline{u}}\left((G_{2}\tensor\one)a^{*\mathrm{t}}\right)\right]\\
 & =(G_{1}\tensor\one)P_{\overline{u}}\left((G_{2}\tensor\one)\Xi(a^{*\mathrm{t}})\right)=(G_{1}\tensor\one)P_{\overline{u}}\left((G_{2}\tensor\one)(A\tensor S)\Xi(a)\right)
\end{split}
\]
for all $a\in M_{n}\odot N$. Therefore
\[
(G_{1}\tensor\one)P_{\overline{u}}(G_{2}\tensor\one)(A\tensor S)\subseteq(A\tensor S)P_{u}.
\]
As $S=J\nabla^{1/2}=\nabla^{-1/2}J$, we have
\[
(G_{1}\tensor\one)P_{\overline{u}}(G_{2}\tensor\one)(A\tensor\nabla^{-1/2}J)\subseteq(A\tensor\nabla^{-1/2}J)P_{u},
\]
and thus
\[
P_{\overline{u}}(\one\tensor\nabla^{-1/2})\subseteq(\one\tensor\nabla^{-1/2})(G_{1}^{-1}\tensor\one)(A\tensor J)P_{u}(A\tensor J)(G_{2}^{-1}\tensor\one).
\]
Recall that $\sigma_{t}^{\om}=\Ad{\nabla^{it}}|_{N}\in\mathrm{Aut}(N)$
for all $t\in\R$. Abusing notation slightly and letting $\sigma_{t}^{\om}:=\Ad{\nabla^{it}}\in\mathrm{Aut}(B(\H))$
for $t\in\R$, we deduce that
\begin{equation}
P_{\overline{u}}\in D(\i\tensor\sigma_{-i/2}^{\om})\quad\text{and}\quad(\i\tensor\sigma_{-i/2}^{\om})(P_{\overline{u}})=(G_{1}^{-1}\tensor\one)(A\tensor J)P_{u}(A\tensor J)(G_{2}^{-1}\tensor\one)\label{eq:invariance_under_mod_aut_gr}
\end{equation}
(see \citep[Theorem 6.2]{Cioranescu_Zsido__analytic_gen}). As $P_{u},P_{\overline{u}}$
are selfadjoint, using \prettyref{eq:invariance_under_mod_aut_gr}
twice we obtain $P_{\overline{u}}\in D(\i\tensor\sigma_{i/2}^{\om})$
and
\begin{multline*}
(\i\tensor\sigma_{i/2}^{\om})(P_{\overline{u}})=\left[(\i\tensor\sigma_{-i/2}^{\om})(P_{\overline{u}})\right]^{*}=(G_{2}^{-1*}\tensor\one)(A\tensor J)P_{u}(A\tensor J)(G_{1}^{-1*}\tensor\one)\\
\begin{split} & =(G_{2}^{-1*}\tensor\one)(A\tensor J)(A\tensor J)(G_{1}\tensor\one)(\i\tensor\sigma_{-i/2}^{\om})(P_{\overline{u}})(G_{2}\tensor\one)(A\tensor J)(A\tensor J)(G_{1}^{-1*}\tensor\one)\\
 & =(G_{2}^{-1*}G_{1}\tensor\one)(\i\tensor\sigma_{-i/2}^{\om})(P_{\overline{u}})(G_{2}G_{1}^{-1*}\tensor\one).
\end{split}
\end{multline*}
This entails that $P_{\overline{u}}\in D(\i\tensor\sigma_{i}^{\om})$
and
\[
(\i\tensor\sigma_{i}^{\om})(P_{\overline{u}})=(G_{2}^{-1*}G_{1}\tensor\one)P_{\overline{u}}(G_{2}G_{1}^{-1*}\tensor\one).
\]
Since $u^{*\mathrm{t}}$ is a (normally not unitary) finite-dimensional
co-representation of $\HH$, from the proof of \citep[Proposition 6.4]{Maes_van_Daele__notes_CQGs}
follows the existence of a positive definite $F\in GL_{n}$ such that
$w:=(F^{1/2}\tensor\one)u^{*\mathrm{t}}(F^{-1/2}\tensor\one)$ is
a unitary co-representation of $\HH$ (which is equivalent to $\overline{u}$)
and $(\i\tensor\sigma_{i}^{h})w=(F\tensor\one)w(F\tensor\one)$. Therefore,
replacing $\overline{u}$ by an equivalent unitary co-representation
of $\HH$ if needed, we may assume that $E=F^{1/2}$ in the notation
of the proof of \prettyref{prop:A_u}, \prettyref{enu:A_u__1}. Thus
$G_{1}=F^{1/2}$ and $G_{2}=F^{3/2}$, so that $G_{2}^{-1*}G_{1}=F^{-1}=(G_{2}G_{1}^{-1*})^{-1}$.
As a result,
\[
(\i\tensor\sigma_{i}^{\om})(P_{\overline{u}})=(F^{-1}\tensor\one)P_{\overline{u}}(F\tensor\one).
\]
Consequently, $(\i\tensor\sigma_{t}^{\om})(P_{\overline{u}})=(F^{it}\tensor\one)P_{\overline{u}}(F^{-it}\tensor\one)$
for every $t\in\R$. Thus, for every $a\in M_{n}\odot N$,
\begin{multline*}
\Xi E_{\overline{u}}((\i\tensor\sigma_{t}^{\om})(a))=P_{\overline{u}}(\one\tensor\nabla^{it})\Xi(a)\\
=(\one\tensor\nabla^{it})(F^{-it}\tensor\one)P_{\overline{u}}(F^{it}\tensor\one)\Xi(a)=\Xi(\i\tensor\sigma_{t}^{\om})((F^{-it}\tensor\one)E_{\overline{u}}((F^{it}\tensor\one)a)),
\end{multline*}
and by injectivity of $\Xi$, $(F^{it}\tensor\one)E_{\overline{u}}((\i\tensor\sigma_{t}^{\om})(a))=(\i\tensor\sigma_{t}^{\om})E_{\overline{u}}((F^{it}\tensor\one)a)$.
Hence $(\i\tensor\sigma_{t}^{\om})E_{\overline{u}}(a)=(F^{it}\tensor\one)E_{\overline{u}}((\i\tensor\sigma_{t}^{\om})((F^{-it}\tensor\one)a))$
for all $a\in M_{n}\odot N$, and so $\sigma_{t}^{\om}(A_{\overline{u}})\subseteq A_{\overline{u}}$.
The result follows from \prettyref{prop:CAP_densities}, \prettyref{enu:CAP_densities_4}.\end{proof}
\begin{rem}
The mean mentioned in \prettyref{thm:invariance_under_mod_aut_gr}
has the flavor of a hypertrace in the sense of Connes \citep{Connes__classification_of_inj_factors}.
Obviously, an invariant mean $m$ on $\G$ admitting an $m$-preserving
conditional expectation from $\Linfty{\G}$ onto $C(\HH)$ satisfies
the indicated condition for every $\gamma\in\mathrm{Irred}(\HH)$.
But a conditional expectation from $\Linfty{\G}$ onto $C(\HH)$ normally
does not exist: for instance, taking $\G:=\Z_{+}$ with an action
yielding $\HH\cong\mathbb{T}$, we would be looking for a conditional
expectation from $\ell_{\infty}$ onto $c_{0}$, which does not exist
as $c_{0}$ is not complemented in $\ell_{\infty}$. The existence
of such a conditional expectation is, nevertheless, not necessary
for the assumptions of \prettyref{thm:invariance_under_mod_aut_gr}
to hold. Notice further that we cannot replace $C(\HH)$ by $\Linfty{\HH}$,
because even in the classical theory, the latter cannot normally be
quantum embedded in $\Linfty{\G}$. \end{rem}
\begin{prop}
\label{prop:invariant_mean_H_discrete}The condition of \prettyref{thm:invariance_under_mod_aut_gr}
holds if $\HH$ is co-commutative, i.e., if it is the dual of a discrete
group.\end{prop}
\begin{proof}
Let $m$ be any invariant mean on $\G$. Let $H$ be the discrete
group so that $\HH=\hat{H}$ (\prettyref{exa:basic_LCQGs}, \prettyref{enu:co_comm_LCQG}).
Co-amenability of $\HH$ is equivalent to amenability of $H$. Denote
by $\left(\l_{h}\right)_{h\in H}$ the translation maps on $H$, and
recall that $\Delta_{\hat{H}}(\l_{h})=\l_{h}\tensor\l_{h}$ for every
$h\in H$. Let $K$ stand for the $w^{*}$-compact convex subset of
$\Linfty{\G}^{*}$ consisting of all invariant means on $\G$. Consider
the action of $H$ on $K$ given by $(h\cdot m)(x):=m(\l_{h}x\l_{h^{-1}})$
($h\in H$, $m\in K$, $x\in\Linfty{\G}$). This action is indeed
well defined, as for every $x\in\Linfty{\G}$ and $\theta\in\Lone{\G}$
we have, by left invariance of $m$,
\[
\begin{split}(h\cdot m)\left((\theta\tensor\i)\Delta(x)\right) & =m\left[\l_{h}(\theta\tensor\i)\Delta(x)\l_{h^{-1}}\right]=m\left\{ (\theta\tensor\i)\left[(\one\tensor\l_{h})\Delta(x)(\one\tensor\l_{h^{-1}})\right]\right\} \\
 & =m\left\{ (\theta\tensor\i)\left[(\l_{h^{-1}}\tensor\one)\Delta(\l_{h}x\l_{h^{-1}})(\l_{h}\tensor\one)\right]\right\} \\
 & =m\left\{ \left[\theta(\l_{h^{-1}}\cdot\l_{h})\tensor\i\right]\Delta(\l_{h}x\l_{h^{-1}})\right\} \\
 & =\theta(\l_{h^{-1}}\cdot\l_{h})(\one)m(\l_{h}x\l_{h^{-1}})=\theta(\one)(h\cdot m)(x).
\end{split}
\]
Therefore, the mean $h\cdot m$ is left invariant, and similarly it
is also right invariant. The action of $H$ on $K$ is plainly affine
and separately continuous ($H$ is discrete!). Thus, by Day's fixed
point theorem \citep[Theorem 1.3.1]{Runde__book_amenability}, there
exists $m\in K$ with $h\cdot m=m$ for all $h\in H$, namely $m(\l_{h}x\l_{h^{-1}})=m(x)$
for all $x\in\Linfty{\G}$ and $h\in H$. As $\left(\l_{h}\right)_{h\in H}$
forms a complete family of irreducible unitary co-representations
of $\HH$, we are done.
\end{proof}

Our ultimate purpose of giving a von Neumann algebraic version of
the Jacobs--de~Leeuw--Glicksberg splitting theorem is achieved in
terms of conditional expectations as follows. The precise nature of
the image of $\i-E^{\CAP}$, namely the \emph{weakly mixing operators},
will be the subject of another paper.
\begin{cor}
\label{cor:E_CAP}Assume that either $\G$ is a LCQG or the condition
of \prettyref{thm:invariance_under_mod_aut_gr} holds. Then there
exists a unique $\om$-preserving conditional expectation $E^{\CAP}$
from $N$ onto $N^{\CAP}$. Moreover, $E^{\CAP}$ is faithful and
normal, and denoting by $P^{\CAP}$ the projection of $\H$ onto $\H^{\CAP}$,
we have $\Gamma\circ E^{\CAP}=P^{\CAP}\circ\Gamma$.\end{cor}
\begin{proof}
By virtue of \prettyref{thm:invariance_under_mod_aut_gr_LCQGs} or
\prettyref{thm:invariance_under_mod_aut_gr} as well as \prettyref{prop:CAP_densities},
\prettyref{enu:CAP_densities_3}, the assertion follows from Takesaki's
theorem and its proof (\citep{Takesaki__cond_exp}, see also \citep{Stratila__mod_thy,Takesaki__book_vol_2}).\end{proof}
\begin{rem}
\label{rem:E_CAP_when_N_finite}If $N$ is finite and $\om$ is a
(finite) trace, the result of \prettyref{cor:E_CAP} holds trivially
(by Takesaki's theorem) without any assumption on $\G$.\end{rem}
\begin{problem}
Determine when the condition of \prettyref{thm:invariance_under_mod_aut_gr}
holds. Specifically, does it hold when $\G$ is a LCQG?
\end{problem}

\begin{problem}
Determine whether the generalized Day's fixed point theorem for amenable
LCQGs (see \citep[Th\'{e}or\`{e}me 2.4, (xii)]{Enock_Schwartz__amenable_Kac_alg}
for Kac algebras; the proof for general LCQGs is identical) can be
applied to produce a generalization of \prettyref{prop:invariant_mean_H_discrete}
for every possible $\HH$, or at least when $\HH$ is a Kac algebra
(equivalently, $h$ is tracial).
\end{problem}
\appendix

\section{State-preserving actions of LCQGs}

The result that we prove in this appendix has appeared implicitly
in several publications, as have some similar results. For instance,
\citep[Proposition 2.4]{Vaes__unit_impl_LCQG} involves a $\delta^{-1}$-invariant
weight and a corresponding co-representation; it is asserted that
the proof is similar to that of \citep[Th\'{e}or\`{e}me 2.9]{Enock__sous_facteurs},
although the second condition in \citep[D\'{e}finition 2.7]{Enock__sous_facteurs}
appears to be missing from \citep[Definition 2.3]{Vaes__unit_impl_LCQG}.
As a second example, in the proof of \citep[Theorem 2.11]{Vaes_Vainerman__low_dim_LCQGs},
the authors used (the $\one$-invariant, rather than $\delta^{-1}$-invariant,
version of) \citep[Proposition 4.3]{Vaes__unit_impl_LCQG} to prove
that their operator $\hat{Z}_{1}$ is the canonical implementing unitary
of the action $\hat{\mu}$, without referring to the second condition
in \citep[D\'{e}finition 2.7]{Enock__sous_facteurs}.

To conclude, \prettyref{thm:action_with_invariant_state_LCQG} is
clearly known to the experts, but since we could not find an explicit
reference, we include here the full proof for completeness. The ideas
are by no means new: they are taken from Kustermans and Vaes \citep[\S 5]{Kustermans_Vaes__LCQG_C_star}.
More general statements, related to the foregoing examples, can be
proved in a similar fashion.
\begin{thm}
\label{thm:action_with_invariant_state_LCQG}Let $\G$ be a LCQG,
$N$ a von Neumann algebra, $\om$ a faithful normal state of $N$
and $\a:N\to N\tensorn\Linfty{\G}$ an $\om$-preserving action of
$\G$ on $N$. Then with $\tau$ being the scaling group of $\G$,
we have
\[
(\sigma_{t}^{\om}\tensor\tau_{-t})\circ\a=\a\circ\sigma_{t}^{\om}\qquad(\forall t\in\R).
\]
\end{thm}
\begin{proof}
We fix some notation. Write $\sigma$ for $\sigma^{\om}$. Let $S,R$
stand for the antipode and unitary antipode of $\G$, respectively.
Denoting by $L^{2}(\G)$ the standard representation Hilbert space
of $\Linfty{\G}$, we let $I,L$ be a conjugation and a strictly positive
operator, respectively, over $L^{2}(\G)$, satisfying $ILI=L^{-1}$
and $R(x)=Ix^{*}I$, $\tau_{t}(x)=L^{it}xL^{-it}$ for every $x\in\Linfty{\G}$,
$t\in\R$ \citep{Kustermans_Vaes__LCQG_C_star,Kustermans_Vaes__LCQG_von_Neumann,Van_Daele__LCQGs}.
Denote by $(\H,\i,\Gamma)$ the GNS construction for $(N,\om)$. As
discussed in \prettyref{sec:prelim_notation}, $\a$ is implemented
by an isometric co-representation $U\in B(\H)\tensorn\Linfty{\G}$
given by $((\i\tensor\theta)(U))\Gamma(a)=\Gamma\bigl((\i\tensor\theta)\a(a)\bigr)$
for all $\theta\in\Lone{\G}$, $a\in N$. The right leg of $U$ is
thus characterized by
\begin{equation}
(\om_{\Gamma(a),\Gamma(b)}\tensor\i)(U)=(\om\tensor\i)\left((b^{*}\tensor\one)\a(a)\right)\qquad(\forall a,b\in N),\label{eq:right_leg_U}
\end{equation}
and similarly
\begin{equation}
(\om_{\Gamma(a),\Gamma(b)}\tensor\i)(U^{*})=(\om\tensor\i)\left(\a(b^{*})(a\tensor\one)\right)\qquad(\forall a,b\in N).\label{eq:right_leg_U_star}
\end{equation}
Recall that as $\G$ is a LCQG, $U$ is in fact unitary \citep[Corollary 4.12]{Brannan_Daws_Samei__cb_rep_of_conv_alg_of_LCQGs}.
Kustermans showed in the proof of \citep[Proposition 5.2]{Kustermans__LCQG_universal}
that, since $U$ is a unitary co-representation of $\G$ (i.e. $U\in B(\H)\tensorn\Linfty{\G}$
and $(\i\tensor\Delta)(U)=U_{12}U_{13}$), it satisfies $(\eta\tensor\i)(U)\in D(S)$
and $S((\eta\tensor\i)(U))=(\eta\tensor\i)(U^{*})$ for every $\eta\in B(\H)_{*}$.
(The argument there uses $C^{*}$-algebra language, and treats co-representations
in the sense that $(\Delta\tensor\i)(U)=U_{13}U_{23}$, but the same
reasoning works in our case.) For all $a,b\in N$, taking $\eta:=\om_{\Gamma(a),\Gamma(b)}$
yields
\[
(\om\tensor\i)\left((b^{*}\tensor\one)\a(a)\right)\in D(S)\quad\text{and}\quad S\left[(\om\tensor\i)\left((b^{*}\tensor\one)\a(a)\right)\right]=(\om\tensor\i)\left(\a(b^{*})(a\tensor\one)\right)
\]
by \prettyref{eq:right_leg_U} and \prettyref{eq:right_leg_U_star}.
Since $S=R\circ\tau_{-i/2}$ by definition (see \citep{Cioranescu_Zsido__analytic_gen}
for this terminology), we have $(\om\tensor\i)\left((b^{*}\tensor\one)\a(a)\right)\in D(\tau_{-i/2})$
and $\tau_{-i/2}\left[(\om\tensor\i)\left((b^{*}\tensor\one)\a(a)\right)\right]=R\left[(\om\tensor\i)\left(\a(b^{*})(a\tensor\one)\right)\right]$.
By \citep[Theorem 6.2]{Cioranescu_Zsido__analytic_gen}, we get
\[
(\om\tensor\i)\left((b^{*}\tensor\one)\a(a)\right)L^{-1/2}\subseteq L^{-1/2}I(\om\tensor\i)\left((a^{*}\tensor\one)\a(b)\right)I,
\]
or equivalently, using that $ILI=L^{-1}$ and putting $K:=IL^{1/2}=L^{-1/2}I$,
\begin{equation}
(\om\tensor\i)\left((b^{*}\tensor\one)\a(a)\right)K\subseteq K(\om\tensor\i)\left((a^{*}\tensor\one)\a(b)\right).\label{eq:K_comm}
\end{equation}
Taking adjoints gives
\begin{equation}
(\om\tensor\i)\left(\a(b^{*})(a\tensor\one)\right)K^{*}\subseteq K^{*}(\om\tensor\i)\left(\a(a^{*})(b\tensor\one)\right).\label{eq:K_star_comm}
\end{equation}
The last two equations hold for every $a,b\in N$. If now $a\in D(\sigma_{-i})$
and $b\in D(\sigma_{i})$, then by the Tomita--Takesaki theory and
\prettyref{eq:K_star_comm},
\[
\begin{split}(\om\tensor\i)\left((b^{*}\tensor\one)\a(a)\right)K^{*} & =(\om\tensor\i)\left(\a(a)(\sigma_{i}(b)^{*}\tensor\one)\right)K^{*}\\
 & \subseteq K^{*}(\om\tensor\i)\left(\a(\sigma_{i}(b))(a^{*}\tensor\one)\right)\\
 & \subseteq K^{*}(\om\tensor\i)\left((\sigma_{-i}(a)^{*}\tensor\one)\a(\sigma_{i}(b))\right).
\end{split}
\]
Hence from \prettyref{eq:K_comm},
\[
(\om\tensor\i)\left((b^{*}\tensor\one)\a(a)\right)K^{*}K\subseteq K^{*}K(\om\tensor\i)\left((\sigma_{i}(b)^{*}\tensor\one)\a(\sigma_{-i}(a))\right).
\]
Write $\nabla$ for the modular operator of $(N,\om)$. Since $K^{*}K=L$,
we conclude that
\[
(\om_{\Gamma(a),\Gamma(b)}\tensor\i)(U)L\subseteq L(\om_{\Gamma(\sigma_{-i}^{\om}(a)),\Gamma(\sigma_{i}^{\om}(b))}\tensor\i)(U)=L(\om_{\nabla\Gamma(a),\nabla^{-1}\Gamma(b)}\tensor\i)(U).
\]
Since $\Gamma(D(\sigma_{-i}))$ and $\Gamma(D(\sigma_{i}))$ are cores
of $\nabla$ and $\nabla^{-1}$, we infer that $\nabla\tensor L^{-1}$
commutes with $U$, namely $U(\nabla\tensor L^{-1})\subseteq(\nabla\tensor L^{-1})U$
(e.g. by \citep[Lemma 5.9]{Kustermans_Vaes__LCQG_C_star}). For every
$x\in N$ we deduce that
\[
\begin{split}(\sigma_{t}\tensor\tau_{-t})\a(x) & =(\nabla^{it}\tensor L^{-it})U(a\tensor\one)U^{*}(\nabla^{-it}\tensor L^{it})\\
 & =U(\nabla^{it}\tensor L^{-it})(a\tensor\one)(\nabla^{-it}\tensor L^{it})U^{*}\\
 & =U(\nabla^{it}a\nabla^{-it}\tensor\one)U^{*}=\a(\sigma_{t}(x)).
\end{split}
\]
This completes the proof.\end{proof}
\begin{acknowledgement*}
We are indebted to L.~Zsid\'{o} for sending us the slides of his
presentation \citep{Zsido__splitting_noncomm_dyn_sys}. We are also
grateful to M.~Alaghmandan, Y.~Choi, M.~Ghandehari, E.~Samei,
P.~M.~So{\ldash}tan, E.~Spinu and V.~G.~Troitsky for fruitful
conversations about the content of this paper, to P.~R.~Buckingham
for many suggestions that led to an improved presentation of the material,
and to the referee for his/her comments.
\end{acknowledgement*}
\bibliographystyle{amsplain}
\bibliography{LCQGs_Ergodic_Thy}

\end{document}